\documentclass[11pt]{article}
\usepackage[dvipsnames]{xcolor}
\usepackage{amsfonts,latexsym,amssymb,amsthm,amsmath,graphicx,cases,comment}
\usepackage{paralist}
\usepackage{graphics} %% add this and next lines if pictures should be in esp format
\usepackage{epsfig} %For pictures: screened artwork should be set up with an 85 or 100 line screen
\usepackage{epstopdf}%This is to transfer .eps figure to .pdf figure; please compile your paper using PDFLeTex or PDFTeXify.

\usepackage[titletoc,title]{appendix}
\usepackage{empheq}
\usepackage{cancel}
\usepackage{tikz}
\usetikzlibrary{arrows,shapes}
\usetikzlibrary{decorations.pathmorphing,decorations.pathreplacing}
\usetikzlibrary{calc,patterns,angles,quotes}
\usepackage{float}

\usepackage[colorlinks=true]{hyperref}
\hypersetup{urlcolor=blue, citecolor=red}

\newtheorem{theorem}{Theorem}[section]
\newtheorem{thm}{Theorem}

\newtheorem{lemma}[thm]{Lemma}
\newtheorem{prop}[thm]{Proposition}
\newtheorem{cor}[thm]{Corollary}

\theoremstyle{definition}

\newtheorem{rmk}{Remark}

\newcommand{\be}{\begin{equation}}
\newcommand{\ee}{\end{equation}}
\newcommand{\bsubeq}{\begin{subequations}}
	\newcommand{\esubeq}{\end{subequations}}

\newcommand{\D}[1]{\mathcal{D}(#1)}
\newcommand{\Ls}{L^2(\Omega)}
\newcommand{\p}{A^{1/2}}
\newcommand{\A}{\mathcal{A}}
\newcommand{\E}{\mathcal{E}}
\newcommand{\Ha}{\mathbb{H}_0}
\newcommand{\Hb}{\mathbb{H}_1}
\newcommand{\Hc}{\mathbb{H}_2}
\newcommand{\Han}{_{\tau,0}}
\newcommand{\Hbn}{_{\tau,1}}
\newcommand{\Hcn}{_{\tau,2}}
\newcommand{\cF}{\mathcal{F}}

\newcommand{\bpm}{\begin{pmatrix}}
	\newcommand{\epm}{\end{pmatrix}}

\newcommand{\bbm}{\begin{bmatrix}}
	\newcommand{\ebm}{\end{bmatrix}}

\numberwithin{equation}{section}
\numberwithin{thm}{section}
\numberwithin{rmk}{section}
\newtheorem{clr}[thm]{Corollary}

\allowdisplaybreaks
\newcommand{\thistheoremname}{}

\newtheorem*{genericthm*}{\thistheoremname}
\newenvironment{namedthm*}[1]
{\renewcommand{\thistheoremname}{#1}%
	\begin{genericthm*}}
	{\end{genericthm*}}

\begin{document}
\title{Singular thermal relaxation limit for the Moore-Gibson-Thompson equation arising in  propagation of acoustic waves.}
	%\title{Continuity of Solutions of the MGT Equation with respect to the Relaxation of Time: Linear and Nonlinear Analysis}
	%\author{Jing Zhang \\
	%{\small Department of Mathematics} \\
	%{\small University of Virginia} \\
	%{\small Charlottesville, VA 22903 USA} \and
	%Roberto Triggiani \\
	%{\small Department of Mathematics} \\
	%{\small University of Virginia} \\
	%{\small Charlottesville, VA 22903 USA}
	%}
	%\date{}
	\author{
		%Irena Lasiecka \\
		%Department of Mathematical Sciences\\
		%University of Memphis, TN, USA \\[2mm]
		%IBS, Polish Academy of Sciences \\
		%Warsaw, Poland
		%\and
		 Marcelo Bongarti, Sutthirut Charoenphon and Irena Lasiecka\\
		{\small Department of Mathematical Sciences}\\
		{\small University of Memphis}\\
		{\small Memphis, TN 38152 USA}\\
	 {\small IBS, Polish Academy of Sciences, Warsaw.}
}

	\date{}
	\maketitle

	%The abstract of your paper
	\begin{abstract}
	Moore-Gibson-Thompson (MGT) equations,  which describe acoustic waves in a  heterogeneous  medium, are considered. These are the  third order in time  evolutions of a predominantly hyperbolic type. MGT models account for a  finite speed propagation due to the appearance of  thermal relaxation  coefficient $\tau > 0 $ in front of the third order time derivative. Since the values of $\tau$ are relatively  small and often negligible, it is important to understand  the asymptotic behavior and characteristics  of the model when $\tau \rightarrow 0 $. This is a particularly delicate issue since the $\tau-$ dynamics is governed by a generator which is singular as $\tau\rightarrow 0.$ It turns out that the limit dynamics corresponds to the linearized Westervelt equation which is of a parabolic type. 
	In this paper, we provide a rigorous analysis of the asymptotics which includes strong convergence of the corresponding evolutions over infinite horizon.  This is obtained by studying convergence rates along with the uniform exponential stability of the third order evolutions. Spectral analysis  for the MGT-equation  along with a  discussion   of spectral  uppersemicontinuity  for  both equations  (MGT and  linearized Westervelt)  will also be provided.
	\end{abstract}

	
	\section{Introduction}
	\hspace{.5cm} In this paper, we  consider PDE system describing the propagation of acoustic waves in a heterogeneous medium. The corresponding models,  referred as Westervelt, Kuznetsov or Moore-Gibson-Thompson equation (MGT), have attracted considerable attention triggered  by important applications in medicine, engineering and life sciences (see \cite{coulo,crig,ham,jordan2,kalt,mg1}).
	Processes such as welding, lithotripsy or high frequency focused  ultrasound depend on accurate modeling involving acoustic equations. 
	From a mathematical point of view, these are either second-order-in -time equations with strong diffusion or third-order-in-time  dynamics. While in the first case the equation is of strongly parabolic type (diffusive effects are dominant),  in the second case the system displays  partial hyperbolic effect which can be easily attested by spectral analysis. 
	From a physical point of view, the difference between two types manifests itself by accounting for finite speed of propagation  for the MGT equation vs infinite speed of propagation for diffusive phenomena.  By accounting for thermal  relaxation in the process,  MGT equation resolves infinite speed of propagation paradox associated with Westervelt-Kuznetsov equation.  The goal of this work is a careful asymptotic analysis  of MGT equation with respect to vanishing relaxation parameter. We will show that the Westervelt-Kuznetsov equation is a limit (in  terms of projected semigroups)
of MGT equation, when the relaxation parameter vanishes. A quantitative  rate of convergence of the corresponding solutions will be derived as well.

	\subsection{Physical motivation, modeling and thermal relaxation parameter}
	
	\hspace{.5cm} Physical models for nonlinear acoustics depend on what constitutive law we choose to describe the dynamics of the heat conduction. 
	According to Fourier (classical continuum mechanics), the dynamics of the thermal flux in a homogeneous and isotropic thermally conducting medium obeys the relation \begin{equation}\overrightarrow{q} = -K\nabla \theta, \label{FL}\end{equation} where $\overrightarrow{q}$ is the flux vector, $\theta = \theta(t,x)$ is the absolute temperature and the constant $K > 0$ is  the thermal conductivity, see \cite{HCP} for more details.
		
		Along with the conservation of mass, momentum and energy, the use of Fourier's law for the heat flux lead us to obtain a number of known equations among which we find the classic second order (in time) nonlinear Westervelt's equation for the acoustic pressure $u=u(t,x)$ which can be written as \begin{equation}\label{WE}
		(1-2ku)u_{tt} -c^2\Delta u - \delta\Delta u_t = 2k(u_t)^2,
		\end{equation} where $k$ is a parameter that depends on the mass density and the constant of the nonlinearity of the medium and $c$ and $\delta$ denote the speed and diffusivity of the sound, respectively. There  are many references addressing various modeling aspects. Within the context of this paper, we refer
		to \cite{HCP, jordan2, kalt, WE} and references therein. 
		% We suggest \cite{WE} and the references therein for more information on the derivation of equation \eqref{WE}.
		
	%	Along with the conservation of mass, momentum and energy, the use of Fourier's law for the heat flux lead us to obtain the classic Kuznetsov's equation \begin{equation}\label{KE}
%		\psi_{tt}-c^2\Delta \psi - \delta \Delta \psi_t = \left(\dfrac{1}{c^2}\dfrac{B}{2A}(\psi_t)^2 + |\nabla \psi|^2\right)_t
%		\end{equation} after neglecting terms of third (or higher) order that might appear in varying quantities present in the computation due to the balance equations. In \eqref{KE}, $\psi = \psi(t,x)$ stands for the acoustic velocity potential, $c$ and $\delta$ denote the speed and diffusivity of the sound respectively and $B/A$ is a parameter of nonlinearity which is material-dependent and it is well known for several types of medium, e.g., blood, muscle, water, some gases, etc. From the Kuznetsov's equation one can derive a number of other known equations among which we find the (also classic) Westervelt's equation which can be equivalently written as \begin{equation}
%		(1-2ku)u_{tt} -c^2\Delta u - b\Delta u_t = 2k(u_t)^2,
%		\end{equation}  obtained by basically neglecting the local nonlinear effect of the gradient term in \eqref{KE}. We suggest \cite{WE} and the references therein for more information.

Unfortunately, Fourier's law does not fully describe the heat diffusion process. Physically, Fourier's law predicts the propagation of the thermal signals at infinite speed, which is unrealistic (see \cite{ap2}). Mathematically, the so-called \emph{paradox of infinite speed of propagation} (or \emph{paradox of heat conduction}, in physics) intuitively means that initial data has an instantaneous effect on the entire space. Quantitatively we translate this notion in terms of support. 

In order to make the notion clear, consider the linearized  homogeneous Westervelt's equation \begin{equation}\label{LWE}
		\alpha u_{tt} -c^2\Delta u - \delta\Delta u_t = 0,
		\end{equation}  ($\alpha$  being a real constant) with initial conditions $u(0,x) = u_0(x)$ and $u_t(0,x) = u_1(x)$. We can simply assume, for the time being, that $x \in \mathbb{R}^n$.
		
		Suppose that $\mbox{supp}(u_0) \cup \mbox{supp}(u_1) \subset B(z,R)$, that is, $u_0$ and $u_1$ have supports inside the ball of radius $R$ and center $z \in \mathbb{R}^n.$ We say that the Partial Differential Equation (PDE) above has \emph{finite speed of propagation} if the solution $u$ is such that $u(t,\cdot)$ has compact support for every $t>0$. More precisely we call \emph{speed of propagation} the number $C$ defined as the infimum of the values $c>0$ such that $\mbox{supp}(u(t, \cdot)) \subset B(z,R+ct).$  If there is no finite $c$ with the above property, we say that the PDE has {infinite speed of propagation}.
		
		One can fairly easily see why Fourier's law leads to infinite speed of propagation. In general lines, neglecting internal dissipation and all sort of thermal sources, the authors of \cite{mech,HCP} used the equation of continuity \begin{equation*}
		\dfrac{\partial \rho}{\partial t} + \nabla \cdot (\rho \overrightarrow{u}) = 0,
		\end{equation*} (here $\overrightarrow{u}$ is the velocity vector of the material point, $t$ is time and $\rho$ is the mass density) to write the balance law for internal heat energy as  \begin{equation}\label{BEq}
		\rho C_p \dfrac{D\theta}{Dt}+\nabla \cdot \overrightarrow{q} = 0,
		\end{equation} where $C_p$ is the specific heat at constant pressure and the operator $$\dfrac{D}{Dt} \equiv \frac{\partial}{\partial t} + \overrightarrow{u} \cdot \nabla$$ represents the material derivative.
		
		Therefore, by simply replacing the flux vector in \eqref{BEq} by the Fourier's law and using the definition of the material derivative we end up with what is called \emph{heat transport equation} \begin{equation}\label{HTE}
		\theta_t + \overrightarrow{u} \cdot \nabla \theta - \eta\nabla^2 \theta = 0,
		\end{equation} where $\eta = \frac{K}{\rho C_p}$ is the constant of thermal diffusivity. Assuming for one moment that the material point does not move (i.e., $\overrightarrow{u} = 0$) the heat transport equation \eqref{HTE} reduces to the classical diffusion equation which is a PDE with parabolic behavior. The solution for the heat equation, as we know, is given by a convolution $u = \phi \star u_0$ where $\phi$ is the fundamental solution of the Laplace equation and $u_0$ is the initial data. From this structure, we can see that small disturbance on the initial data has the potential to affect the whole solution in the entire space.
		
		In order to address this \emph{defect} and account for  a  finite speed to the heat conduction, several improvements and modifications to the Fourier's law were studied (see \cite{ap3}). Although different, all the modifications agree with the fact that it is unrealistic to consider that any change of temperature is immediatly felt regardless of position.
		It is interesting to note that  the first work to notice this phenomena with a  derivation of  a new  third order in time model is \cite{stokes} by Professor G. G. Stokes.
		 After 97 years, in \cite{ap5}, C. Cattaneo derived what today is known as the Maxwell-Cattaneo law (see also \cite{ap2,ap3,ap4}). 
	
		The Maxwell-Cattaneo Law is given by \begin{equation} \label{MC}
		\overrightarrow{q} + \tau \dfrac{\partial \overrightarrow{q}}{\partial t} = -K\nabla \theta,
		\end{equation} and managed to remove the infinite speed paradox by adding the so called \emph{thermal inertia} term which is proportional to the time derivative of the flux vector. 
		
		It is important to observe that Maxwell-Cattaneo law as we presented in \eqref{MC} resolves the paradox of infinite speed of propagation, but the diffusion process is only free of paradoxes in the case where the body (or the object) of the dynamic is resting. In the moving frame, this same constitutive law gives rise to another paradox related to the Galilean relativity regarding the invariance of physical laws in all frames. This latter and last paradox can be resolved by replacing the time derivative in \eqref{MC} by the material derivative. More details about this issue can be found in \cite{HCP}.
		
		The material-dependent constant $\tau$ is known as the \emph{thermal relaxation parameter} or \emph{time relaxation parameter} and is the center of this paper. Physically $\tau$ represents the time necessary to achieve steady heat conduction once a temperature gradient is imposed to a volume element. This time lag can be (and in fact it is) translated to different phenomena and contexts, as is the case where the models are used to study problems of High-Frequency Ultrasound (HFU) in lithotripsy, thermotherapy, ultrasound cleaning and sonochemistry. See \cite{jmgt,kalt,HCP}.
		
	The goal of this paper is  to \emph{quantify} the sensitivity the thermal relaxation parameter $\tau$ on a variety of materials by studying a singular perturbation problem, which makes sense since a number of experiments found this parameter to be small in several mediums, although not all. Among the ones where $\tau$ is not small we find biological tissue (1-100 seconds), sand (21 seconds), H acid (25 seconds) and NaHCO$_3$ (29 seconds) (see \cite{HCP}). Among the ones with $\tau$ small we find cells and melanosome (order of milliseconds), blood vessels (order of microseconds, depending on the diameter) (see \cite{RelTime1}) and most metals (order of picoseconds) (see \cite{HCP}).
		
		The same procedure as to obtain the Westervelt's equation leads us now to the third-order (in time) nonlinear Moore-Gibson-Thompson (MGT) equation \begin{equation}\label{JMGT}
		\tau u_{ttt} + (1-2ku)u_{tt} - c^2\Delta u - b \Delta u_t = 2k(u_t)^2,
		\end{equation} where $k$ and $c$ has the same meaning as the ones in the Westervelt's equation but the diffusivity of the sound $\delta$ also suffers a change due to the presence of the thermal relaxation parameter $\tau$ and gives place to a new parameter $b = \delta + \tau c^2.$ The operator $\Delta$ is understood as the Laplacian subject to suitable boundary conditions-Dirichlet, Neumann or Robin. It should be noted that the original version of this model dates back to Stokes paper \cite{stokes}.
	A typical JMGT equation is equipped with the  additional more precise  physical parameters \cite{jmgt}, however, for the sake of transparency  only  the canonical abstract form is retained.
		
		 \subsection{Past literature and introduction of the problem}
		 	
		 	\hspace{.5cm} This section  collects the  relevant past   results pertinent to the model  under consideration.
			 %about the model we are working on available in the literature.
		 	
		Let $\Omega$ be a bounded domain in $\mathbb{R}^n$ ($n=2,3$) with a $C^2-$boundary $\Gamma = \partial \Omega$ immersed in a resting medium. We work with $\Ls$ but the treatment could be similarly carried out to any (separable)  Hilbert space $H$.  Consider $A: \D{A} \subset \Ls \to \Ls$ defined as the Dirichlet Laplacian, i.e., $A = -\Delta$ with $\D{A} = H_0^1(\Omega)  \cap H^2(\Omega)$. All the results remain true if we assume $A$ to be any unbounded positive self-adjoint operator with compact resolvent  defined on $H$. 
		 	
		 	Consider the nonlinear third order evolution  \begin{equation}\label{MGTnl}\begin{cases}
		 	\tau u_{ttt}+(1+2ku)u_{tt} + c^2 A u + b A u_t = 0, \\
		 	u(0,\cdot) = u_0, u_t(0,\cdot) = u_1, u_{tt}(0,\cdot) = u_2,
		 	\end{cases}
		 	\end{equation} and its  linearization 
		 	\begin{equation}\label{MGTl}\begin{cases}
		 	\tau u_{ttt}+\alpha u_{tt} + c^2 A u + b A u_t = 0, \\
		 	u(0,\cdot) = u_0, u_t(0,\cdot) = u_1, u_{tt}(0,\cdot) = u_2.
		 	\end{cases}
		 	\end{equation}
			The natural phase spaces associated with these evolutions are  the following: 
			\begin{eqnarray}
			\Ha \equiv \mathcal{D}(A^{1/2}) \times \mathcal{D}(A^{1/2}) \times \Ls  ~~and ~~\Hb \equiv \mathcal{D}(A) \times \mathcal{D}(A^{1/2}) \times \Ls. \end{eqnarray}

		 	Generation of linear semigroups associated with (\ref{MGTl})  has been studied in \cite{mgtp1,TrigMGT}  where it was shown that for any $\tau >0, b > 0$  (\ref{MGTl}) generates a strongly continuous  group on either $\Ha$ or $\Hb$. This result depends on $b > 0$. When $b =0$  the generation of semigroups fails \cite{fatto}.
			
			The nonlinear (quasilinear)  model (\ref{JMGT}) has been treated in  \cite{jmgt} where it was shown that for the  initial data sufficiently small  in $\Hb$, i.e., in  a ball $B_{\Hb}(r) $  there  exists  nonlinear semigroup   operator defined on $\Hb$ for all $ t > 0$.  The value of $r$ depends only on the values of the physical  parameters in the equation and not on $t > 0 $. The aforementioned result depends on  uniform stability of the dynamics of \eqref{MGTl} and this holds for $\gamma = \alpha - \tau c^2 b^{-1} > 0 $. 
			
%			 the authors showed that both linear and nonlinear problems above are well-posed. Fixed point theory was used in \cite{jmgt} to show well-posedness of the nonlinear problem. In \cite{mgtp1} was showed that \eqref{MGTl} generates a $C_0-$group in two different spaces.
			 
			  Subsequently,  the authors in \cite{TrigMGT} showed that the linear equation generates a $C_0-$group in four different spaces with exponential stability 
			   provided $\gamma = \alpha - \tau c^2 b^{-1} > 0.$ In case $\gamma = 0$ the system is conservative and in case $\gamma < 0$, by assuming very regular energy spaces the authors in \cite{liz1} showed that \eqref{MGTl} generates a chaotic semigroup.
		 	
		 	Spectral analysis  of the linear problem was also studied \cite{TrigMGT,jmgt,mgtp1}.   The spectrum consists of continuous spectrum and point spectrum. 
			The location of the eigenvalues  confirms partially  hyperbolic character of the dynamics. 
		 	
		 	The same  model with added  memory,  where the latter accounts for molecular relaxation, was considered in \cite{mem1,mem2,mem3,mem4,mem5} for linear case and in \cite{mem6} for the nonlinear case.
		 	
		 	%{\bf \color{red} There must be many more works but I do not know about them. Maybe Dr. Lasiecka can complete the list}
		 	
		 	All the results obtained and mentioned above pertain to the situation when $\tau > \tau_0 > 0$. Since the parameter $\tau$ in many applications  is typically very small, it is essential  to understand the effects of diminishing values of relaxation parameter on quantitative properties of the underlined dynamics. This will provide important  information 
			on sensitivity of the model with respect to time relaxation. 
			The goal of this paper is precisely to consider the vanishing parameter $\tau \rightarrow 0 $ and its consequences on the  resulting evolution. 
			Specific  questions we ask are the following: 
			\begin{itemize}
			\item[]
			${\bullet}$ Convergence of semigroups with respect to vanishing relaxation parameter $\tau \geq 0 $. \\
			%\item
			$\bullet$ Uniform (with respect to  $\tau > 0$) asymptotic  stability  properties of  the  ``relaxed" groups.\\
			%\item
			$\bullet$ Asymptotic (in $\tau$) behavior of the spectrum for the family of  the generators.
			\end{itemize}

			To our best knowledge,  this is the first work that takes into consideration  asymptotic properties of the MGT  dynamics  with respect to the vanishing relaxation parameter. 
			The limiting evolution  changes the  character from a  hyperbolic group  to a  parabolic semigroup. This change  is expected to be reflected by the  asymptotic properties of the spectrum and quantitative  estimates  for the corresponding evolutions. It should also be noted that the problem under consideration does not fit the usual Trotter-Kato type of the framework. This is due to the fact that the limit problem corresponds formally to degenerated structure. Thus, convergence of the resolvents (condition required by Trotter Kato framework)  does not have a  natural interpretation.  
		 	
	\section{Main Results}
	
	\subsection{Convergence of  the  projected semigroup solutions}
	
	\hspace{.5cm} As before, let $\Omega$ be a bounded domain in $\mathbb{R}^n$ ($n=2,3$) with a $C^2-$boundary $\Gamma = \partial \Omega$ immersed in a resting medium and $A: \D{A} \subset \Ls \to \Ls$ defined as the Dirichlet Laplacian, i.e., $A = -\Delta$ with $\D{A} = H_0^1(\Omega) \cap H^2(\Omega)$. 
	
	Let $T > 0$.  We consider a  family of ``hyperbolic" abstract third order  problems
	 \begin{equation}\label{MGTtau}\begin{cases}
	 \tau u^\tau_{ttt}+\alpha u^\tau_{tt} + c^2 A u^\tau + b^\tau A u^\tau_t = 0,  \ t > 0, \\
	 u^\tau(0,\cdot) = u_0, u^{\tau}_t(0,\cdot) = u_1, u^{\tau}_{tt}(0,\cdot) = u_2,
	 \end{cases}
	 \end{equation} where $b^\tau = \delta + \tau c^2$ and $\alpha, c, \delta,\tau >0.$
	 
	 We rewrite \eqref{MGTtau} abstractly by using a mass operator $M_\tau$ as below:  \begin{equation}\begin{cases}M_\tau U^\tau_t(t) = \mathcal{A}_0^\tau U^\tau(t), \ t>0, \\ U^\tau(0)=U_0=(u_0,u_1,u_2)^T, \end{cases}\label{AbS}\end{equation}
	 or equivalently with $\A^{\tau} = M_{\tau}^{-1} \A_0^{\tau} $ 
	  \begin{equation}\begin{cases} U^\tau_t(t) = \mathcal{A}^\tau U^\tau(t), \ t>0, \\ U^\tau(0)=U_0=(u_0,u_1,u_2)^T, \end{cases}\label{AbS}\end{equation}

	  where \begin{equation}U^\tau \equiv \left(\begin{array}{c}
	 u^\tau \\
	 u^\tau_t\\
	u^\tau_{tt}
	 \end{array}\right); \ \mathcal{A}^\tau\equiv \left(\begin{array}{ccc}
	 0  & 1 & 0 \\
	 0 & 0 & 1 \\
	 - \tau^{-1} c^2A & -\tau^{-1}b^\tau A & -\tau^{-1}\alpha 
	 \end{array}\right); \ M_{\tau}\equiv \left(\begin{array}{ccc}
	 1  & 0 & 0 \\
	 0 & 1 & 0 \\
	 0 & 0 & \tau
	 \end{array}\right)\label{op1}.\end{equation}
	 The evolution described in (\ref{AbS}) can be considered on several  product spaces with the results depending on the space and the domain where $\A^{\tau} $ is defined.
	 \begin{rmk}
	 	The generator $\A^\tau$ ``blows up'' when $\tau\rightarrow 0.$
	 \end{rmk}
	 
	% The first two results  establish the  generation of a semigroup associated with  \eqref{MGTtau} and the equi-boundedness (in $\tau$) of the associated family of solutions. The well posedness results are direct consequence of the works  \cite{mgtp1} and \cite{jmgt} (see also \cite{TrigMGT}) and some topological scaling. 
	 The following three spaces are important for the development for our result. We define $\mathbb{H}_0, \mathbb{H}_1, \mathbb{H}_2$ as $$\mathbb{H}_0 \equiv \D{\p} \times \D{\p} \times L^2(\Omega);$$ $$\mathbb{H}_1 \equiv \D{A} \times \D{\p} \times L^2(\Omega);$$ $$\mathbb{H}_2 \equiv \D{A} \times \D{A} \times \D{\p}.$$
	 
	 The operators $\A^{\tau}$ are considered on each of these spaces with natural domains induced by the given  topology. 
	  For instance, $\A^{\tau} : \mathcal{D}(\A^{\tau} ) \subset \Ha  \rightarrow \Ha $ has the domain defined by 
	  $$\mathcal{D} (\A^{\tau}) = \{ (u,v,w) \in \Ha; c^2u + b^{\tau} v \in \mathcal{D}(A) \}. $$
	  Clearly, the domains are not compact in $\Ha$. Analogous setups are made for $\Hb$ and $\Hc$.
	  
	  For each $\tau >0$,  we will consider  weighted norms defined by the means of the mass operator  $M_{\tau} $.
	 $$|| M_{\tau}^{1/2} U||^2_{\Ha},  || M_{\tau}^{1/2} U||^2_{\Hb},||M_{\tau}^{1/2} U||^2_{\Hc},$$ that is, $$\|(u,v,w)\|\Han^2 = \|u\|_{\D{\p}}^2 + \|v\|_{\D{\p}}^2 + \tau\|w\|_{2}^2 = ||M_{\tau}^{1/2} U ||^2_{\Ha};$$
	 $$\|(u,v,w)\|\Hbn^2 = \|u\|_{\D{A}}^2 + \|v\|_{\D{\p}}^2 + \tau\|w\|_{2}^2 = ||M_{\tau}^{1/2}U||^2_{\Hb};$$
	 $$\|(u,v,w)\|\Hcn^2 = \|u\|_{\D{A}}^2 + \|v\|_{\D{A}}^2 + \tau\|w\|_{\D{\p}}^2 = ||M_{\tau}^{1/2} U||^2_{\Hc}$$ with $\|\cdot\|_2$ representing the standard $L^2-$norm.
	  We shall also use the rescaled  notation: $\Ha^{\tau}  = M_{\tau}^{1/2} \Ha$,  $ \Hb^{\tau}  = M_{\tau}^{1/2} \Hb$, $\Hc^{\tau}  = M_{\tau}^{1/2} \Hc$ with an obvious interpretation for the composition where  the elements of $\Ha^{\tau} $ coincide with the elements of $\Ha$ and induced topology given by  $\|(u,v,w)\|\Han$.

	 \begin{theorem}\label{wp-apb} \emph{(}{\bf Generation of a group  on $\Ha$ and $\Hc$}\emph{).} Let
	 % $U_0=(u_0,u_1,u_2) \in \mathbb{H}_0$ and
	  $\alpha, c, \delta > 0.$  Then, for each $\tau >0 $ the operator $\mathcal{A}^\tau$ generates a $C_0-$group $\{T^\tau(t)\}_{t \geqslant 0}$ on $\mathbb{H}_0$ and also on $\Hc$.
		\end{theorem}
		Theorem \ref{wp-apb} follows from \cite{TrigMGT} applied to $\Ha$ space.  The invariance of the generator  under the multiplication by fractional powers of $A$  leads to the result stated for   $\Hc$.
		
	\begin{theorem}\label{wp-apb-he} 	
	 	%\emph{(a)} \emph{(}{\bf generation of a group}\emph{)}  For each $\tau >0 $ the operator $\mathcal{A}^\tau$ generates a $C_0-$group $\{T^\tau(t)\}_{t \geqslant 0}$ on $\mathbb{H}_1$ .
	 	\emph{(}{\bf Equi-boundedness and uniform (in $\tau$) exponential stability in $\mathbb{H}_0^{\tau}$}\emph{)}. Consider the family $\mathcal{F} = \{T^\tau(t)\}_{\tau>0}$ of groups generated by $\A^\tau$ on $\Ha.$  Assume that $\gamma^\tau \equiv \alpha - c^2\tau (b^\tau)^{-1}\geq \gamma_0 >0 $. Then, there exists $\tau_0 > 0$ and constants $M = M(\tau_0),\omega = \omega(\tau_0)>0$ $($both independent on $\tau)$ such that $$\|T^\tau(t)\|_{\mathcal{L}(\mathbb{H}^{\tau}_0)} \leqslant M e^{-\omega t} \ \mbox{for all} \ \tau \in (0,\tau_0] \ \mbox{and} \ t \geqslant 0.$$
	 \end{theorem}
 
 \medskip 
 
 \begin{theorem}\label{wp-apb-he-n}Let $\alpha, c, \delta > 0.$ Then
 	
 	\emph{(a)} \emph{(}{\bf generation on $\Hb$}\emph{)}  For each $\tau >0 $ the operator $\mathcal{A}^\tau$ generates a $C_0-$group $\{T^\tau(t)\}_{t \geqslant 0}$ on $\mathbb{H}_1$. 
 	
 	\emph{(b)} \emph{(}{\bf equi-boundedness and uniform (in $\tau$) exponential stability}\emph{)} Consider the family $\mathcal{F}_1 = \{T^\tau(t)\}_{\tau>0}$ of groups generated by $\A^\tau$ on $\Hb.$  Assume $\gamma^{\tau} > \gamma_0  > 0 $. Then, there exists $\tau_0 > 0$ and constants $\overline{M}_1 = \overline{M}_1(\tau_0), \  \overline{\omega}_1 = \overline{\omega}_1(\tau_0)>0$, both independent on $\tau$ such that $$\|T^\tau(t)\|_{\mathcal{L}(\mathbb{H}^{\tau}_i)} \leqslant \overline{M}_1 e^{-\overline{\omega}_1 t} \text\ \mbox{for all} \ \tau \in (0,\tau_0] \ \mbox{and} \ t \geqslant 0.$$
 \end{theorem}

\begin{rmk}\label{rmkh2}
	Notice that the space $\mathbb{H}_2$ is obtained by multiplication of elements in $\mathbb{H}_0$ by $A^{1/2}$ (componentwise), therefore, if we assume initial data in $\mathbb{H}_2$, it follows that uniform (in $\tau$) boundedness and stability of the dynamics remain true.
\end{rmk}
  In order to characterize  asymptotic behavior of  the family $\mathcal{F}$  of the groups $T^{\tau}(t) $  we introduce the space $\Ha^{0} \equiv \mathcal{D}(A^{1/2}) \times \mathcal{D}(A^{1/2} ) $ and 
   the projection operator   $P: \mathbb{H}_0 \to \mathbb{H}_0^0$  defined as $$\mathbb{H}_0 \ni (u,v,w) \mapsto (u,v) \in \mathbb{H}_0^0.$$ 
   With this notation Theorem \ref{wp-apb-he} implies  uniform boundedness of the sequence 
   \begin{equation}\label{lim}
  ||P T^{\tau}(t)E||_{L(\Ha^0)} \leq M e^{-\omega t}, t > 0,
  \end{equation}
   where $E$ denotes the extension operator from $\Ha^0\rightarrow \Ha^{\tau} $  defined by  $E(u,v) \equiv (u,v,0) $. From  (\ref{lim}) we deduce that for every $ U^0=(u_0,u_1) \in \Ha^0 $ 
   the corresponding  projected  solutions $P T^{\tau}(t) EU^0$ have a weakly convergent subsequence in $\Ha^0$ and weakly star in $L^{\infty}(0,\infty; \Ha^0).$  By standard distributional calculus  one shows that 
   such subsequence converges {\it weakly} to  $U^0(t) = (u^0(t), u_t^0(t) )$  which satisfies (distributionally)  the following {\bf limit} equation 
   
	%As typical  in singular perturbation theory, the {\bf limit} problem is  expected  to be the one obtained by setting $\tau=0$ formally:
	  \begin{equation}\label{MGT0}\begin{cases}
 	\alpha u^0_{tt} + c^2 A u^0 + \delta A u^0_t = 0,\\
 	u^0(0,\cdot) = u_0, u^0_t(0,\cdot) = u_1,
 	\end{cases}
 	\end{equation}
 	
 	which rewritten as first order system  becomes 
	 \begin{equation}\begin{cases}U^0_t(t) = \mathcal{A} U^0(t), \ t>0, \\ U^0(0)=U_0^0 = (u_0,u_1)^T, \end{cases}\label{AbS0}\end{equation} where \begin{equation}U^0 \equiv \left(\begin{array}{c}
 	u^0 \\
 	u^0_t
 	\end{array}\right); \ \mathcal{A}=\left(\begin{array}{cc}
 	0  & I \\
 	-c^2 \alpha^{-1}A & -\delta \alpha^{-1}A
 	\end{array}\right)\label{op2}\end{equation}
 	and
	$$\A : \mathcal{D}(\A) \subset \Ha^0  \rightarrow  \Ha^0 $$ with 
	$$\mathcal{D}(\A) =\{ (u,v) \in \mathcal{D}(\A^{1/2} ) \times \mathcal{D}(\A^{1/2} )
 	; c^2 u + \delta v \in \mathcal{D}(A^{3/2} ) \}. $$
	Equation (\ref{MGT0}) is  a known and well studied in the literature  strongly damped wave equations. In fact,  generation of an analytic  and exponentially decaying semigroup on  the space $ \mathcal{D}(A^{1/2}) \times L^2(\Omega)$ is a standard by now result \cite{lunardi,trig89, crd82}. Less standard is the analysis  on $\mathbb{H}_0^0 \equiv \D{\p} \times \D{\p}$, where contractivity and dissipativity are no longer valid. This  latter is the framework relevant to our analysis.

 	\begin{prop}\label{lp} \emph{(a)} \emph{(}{\bf generation of a semigroup on  $\mathbb{H}_0^0$ }\emph{)} Let $ \mathbb{H}_0^0 \equiv \D{\p} \times \D{\p}$ and $\alpha,\delta,c > 0$. Then the operator $\mathcal{A}$ generates an $($noncontractive$)$ analytic semigroup $\left\{T(t)\right\}_{t \geqslant 0}$ in $\mathbb{H}_0^0.$
 		
 		\emph{(b)} \emph{(}{\bf exponential stability}\emph{)} There exist constants $M_0,\omega_0>0$ such that $$\|T(t)\|_{\mathcal{L}(\Ha^0)} \leqslant M_0e^{-\omega_0t}, \ t\geqslant 0.$$
 	\end{prop}
	\begin{proof}
	The well-posedness and analyticity of the associated generator on the space $ L^2(\Omega) \times L^2(\Omega)$  
 %(taking into consideration our future description of the spectrum $\A.$)
  is a direct consequence of \cite{bible} (Theorem 3B.6, p. 293) and  \cite{trigspec} (Proposition 2.2, p. 387). Invariance of the semigroup under the action of $A^{1/2} $  implies the same result  in $\mathbb{H}_0^0 $, hence justifying the  part (a) of  Proposition \ref{lp}. As to the exponential stability,  while this is a well known fact proved by energy methods on the space $\mathcal{D}(A^{1/2}) \times \Ls $, the decay rates on  $\Ha^0$ (nondissipative case)   need a  justification. 
 In our case,  this follows from the estimate  in (\ref{lim})  along with weak lower semicontinuity of $|| P T^{\tau}(t) EU^0||_{\Ha^0}^2 $. The conclusion on exponential stability can also be derived independently of the family $\mathcal{F} $,  by evoking analyticity of the generator \cite{trig89} 
    along with the spectrum growth determined condition and the analysis of the  location of the spectrum (see section 2.2 below).
	\end{proof}

\begin{rmk}
	Proposition \ref{lp} also holds   with $\mathbb{H}_0^0$ replaced by   $\Hb^0 \equiv \D{A} \times \D{A}.$
\end{rmk}

 Our main interest and goal of this work is 	to provide a quantitative description of {\it strong} convergence, when $\tau \rightarrow 0 $,  of hyperbolic  groups $T^{\tau}(t) $ to the parabolic like semigroup $T(t)$.  Our result is formulated below. 
 	
 	\begin{theorem}\label{cr} \emph{(a)} {\bf (Rate of convergence)}  Let $U_0 \in \mathbb{H}_2.$ Then there exists $C = C(T, \tau_0)$ such that  $$\|P T^\tau(t) U_0 - T(t)PU_0\|_{\mathbb{H}_0^0}^2 \leqslant C \tau \|U_0\|_{\mathbb{H}_2}^2$$  uniformly for  $ t \in [0,T].$
 		
 		\emph{(b)} {\bf (Strong convergence)} Let $U_0 \in \mathbb{H}_0$. Then  the following strong convergence holds \begin{equation}\label{convn}
 		\|P T^\tau(t)U_0  - T(t)PU_0\|_{\mathbb{H}_0^0 } \to 0 ~as~\tau \to 0
 		\end{equation} uniformly for all $t \geqslant 0.$
 	\end{theorem}
 
 \begin{rmk}
 	Note that Theorem \ref{cr}  pertains to uniform (in time) strong convergence on infinite time  horizon. This fact is essential  in infinite horizon optimal control theory \cite{bible}.
 \end{rmk}
 \begin{rmk}
A standard tool for proving strong convergence of semigroups is Trotter-Kato Theorem \cite{kato}. However, this approach  does not apply to the problem under consideration due to  the singularity of the  family of generators. A consistency requirement (convergence of  the resolvents) is problematic due to specific framework where  the family of $\A^{\tau} $   becomes singular when $\tau =0 $. More refined approach applicable to this particular framework  will be developed.   
 \end{rmk}

 	The Theorem \ref{cr} provides the  information about  strong  convergence of the solution $u^\tau$ and its first derivative in time.  Regarding the second time derivative, we have the following.

 	 \begin{prop}\label{sd}
 		Let $U_0 \in \mathbb{H}_0$, then we have $$\tau^{1/2} u^\tau_{tt} \to 0 \ \mbox{weakly}^\ast \ \mbox{in} \ L^\infty(0,\infty; \Ls).$$
 	\end{prop}
 	
 	\subsection{ Spectral Analysis and Comparison Between  $\sigma(\A^\tau)$ and $\sigma(\A)$}
 	
 	\hspace{.5cm} Recall that  $A$  is assumed to be a positive self-adjoint operator  with compact resolvent defined on a infinite-dimensional Hilbert space  H ($L^2(\Omega)$ for instance). This allow us to infer that the spectrum of $A$ consists purely of the point spectrum. Moreover, it is countable and positive. In other words: $$\sigma(A) = \sigma_p(A) = \{\mu_n\}_{n\in \mathbb{N}} \subset \mathbb{R}_+^\ast$$ and $\mu_n \to \infty$ as $n\to \infty.$
 	
 	We begin with the characterization of the  spectrum of $\A,$ the operator corresponding to the limit problem.  See Figure \ref{fig:spec_sorder}.
 	
 		\begin{prop}\label{spec2}\emph{(a)} The residual spectrum is empty: $\sigma_r(\mathcal{A}) = \emptyset.$
 		
 		\medskip
 		
 		\emph{(b)} The continuous spectrum consists of one single real value: $\sigma_c(\mathcal{A}) = \left\{-\dfrac{c^2}{\delta}\right\}.$
 		
 		\medskip
 		
 		\emph{(c)} The point spectrum is given by $$\sigma_p(\mathcal{A}) = \left\{\lambda \in \mathbb{C}; \alpha \lambda^2 + \delta \mu_n \lambda + c^2 \mu_n = 0, \ n\in \mathbb{N}\right\} =\left\{\lambda_n^{0},\lambda_n^{1}, n \in \mathbb{N}\right\},$$ where \emph{Re}$(\lambda_n^{i}) \in \mathbb{R}_-^\ast$ for all $n \in \mathbb{N}$ and $i = 1,2.$ Moreover, both branches are eventually real and the following limits hold: $$\lim\limits_{n\to \infty} \lambda_n^{0} = -\dfrac{c^2}{\delta} \ \mbox{and} \ \lim\limits_{n\to \infty} \lambda_n^{1} = -\infty.$$
 	\end{prop}
 
 \
 	
 	Regarding the spectrum of $\mathcal{A}^\tau$ we have, see Figure \ref{fig:third_fix} for each $\tau > 0.$
 	 
 	 \medskip
 	
 	\begin{prop}\label{spec1} \emph{(a)} The residual spectrum is empty $\sigma_r(\mathcal{A}^\tau) = \emptyset$ for all $\tau >0.$
 		
 		\medskip
 		
 		\emph{(b)} The continuous spectrum is either empty or consists of a single real value: $$\sigma_c(\mathcal{A}^\tau) = \begin{cases}
 		\left\{-\dfrac{c^2}{b^\tau}\right\} \ &\mbox{if} \ \gamma^\tau > 0, \\
 		\emptyset \ &\mbox{if} \ \gamma^\tau = 0,
 		\end{cases}$$
 		where $\gamma^\tau \equiv \alpha - c^2\tau (b^\tau)^{-1}.$
 		
 		\medskip
		%\cite{las-jee}
 		
 		\emph{(c)} The point spectrum is given by $$\sigma_p(\mathcal{A}^\tau) = \{\lambda  \in \mathbb{C}; \ \tau\lambda^3 + \alpha \lambda^2 + b^\tau \mu_n \lambda + c^2 \mu_n = 0, \ n \in \mathbb{N}\} = \left\{\lambda_n^{0,\tau},\lambda_n^{1,\tau},\lambda_n^{2,\tau}, n \in \mathbb{N}\right\}.$$ One of the branches, say $\lambda_n^{0,\tau}$, is eventually real while the other two branches are eventually complex, conjugate of each other and the following limits hold: $$\lim\limits_{n\to \infty} \lambda_n^{0,\tau} = -\dfrac{c^2}{b^\tau}, \ \lim\limits_{n\to \infty} \mbox{\emph{Re}}(\lambda_n^{1,\tau}) = -\dfrac{\gamma^\tau}{2\tau} \ \mbox{and} \lim\limits_{n\to \infty} |\mbox{\emph{Img}}(\lambda_n^{1,\tau})| = \infty.$$
 	\end{prop}
 \begin{figure}[H]
 	\begin{center}
 		\includegraphics[scale=.15]{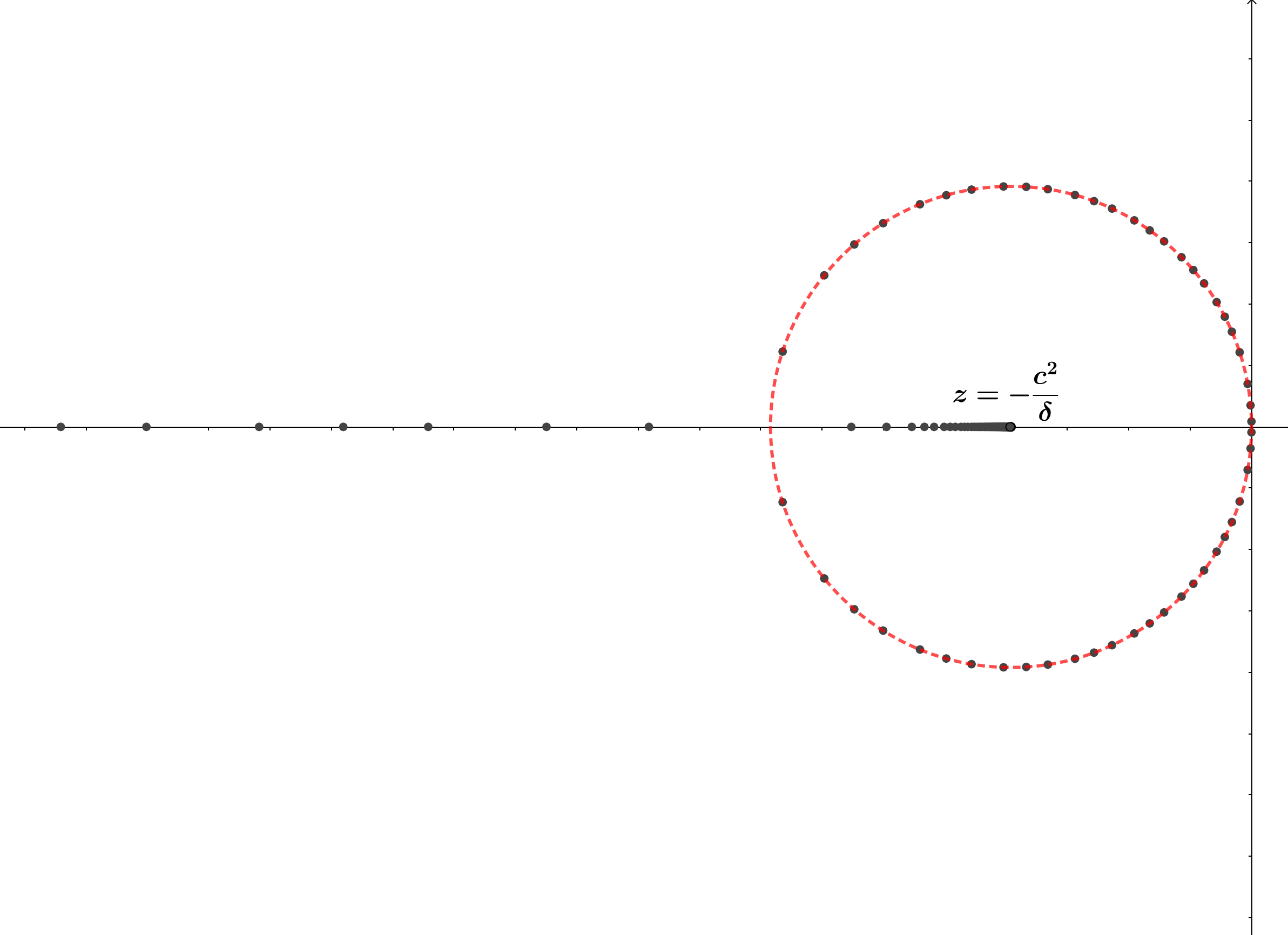}
 		\caption{{\bf Graphical representation of $\sigma_p(\A)$ and $\sigma_c(\A)$}: the circle in red is centered at $\left(-\frac{c^2}{\delta},0\right)$ and has radius $r=\frac{c^2}{\delta}$. The shape of the eigenvalues is represented by the curve of black dots. Clearly, the bigger is the radius, the bigger is number of complex roots. Nevertheless, the asymptotic behavior of the eigenvalues are the same: when $n$ becomes sufficiently large, all the eigenvalues are real and we have one branch converging to $-\infty$ while the other converge to the point in the continuous spectrum $z = -\frac{c^2}{\delta}.$} 
 		\label{fig:spec_sorder}
 	\end{center}
 \end{figure}
The next lemma allows us to establish quantitative relation between $\sigma(\mathcal{A})$ and $\sigma(\mathcal{A}^\tau)$ for small $\tau.$
		\begin{lemma}\label{t1}
			Let $n \in \mathbb{N}$ fixed. Then, among the three roots $\{\lambda_n^{0,\tau}, \lambda_n^{1,\tau}, \lambda_n^{2, \tau}\}_{\tau >0}$ of the equation $$\tau \lambda^3 + \alpha \lambda^2 + b^\tau \mu_n \lambda + c^2 \mu_n = 0,$$ there are two converging, as $\tau \to 0$, to the two roots $\{\lambda_n^1, \lambda_n^2\}$ of the equation $$\alpha \lambda^2 + \delta \mu_n \lambda + c^2\mu_n = 0.$$
		\end{lemma}
	\begin{proof}
			Write $$
			\tau \lambda^3 + \alpha \lambda^2 + b^\tau \mu_n \lambda + c^2 \mu_n = p(\lambda)(\alpha \lambda^2 + \delta \mu_n \lambda + c^2 \mu_n) + q_\tau(\lambda)$$ with $$p(\lambda) = \dfrac{1}{\alpha^2}\left[\tau\alpha\lambda + \alpha^2 - \tau\delta\mu_n\right]$$ $$q_\tau(\lambda) = \dfrac{\tau \mu_n}{\alpha^2}\left[\alpha(\alpha-1)c^2 - \delta^2\mu_n\right]\lambda+ \dfrac{\tau \mu_n^2 c^2 \delta}{\alpha^2}$$ and notice that $$\lim_{\tau \to 0} q_\tau(\lambda) = 0 \ \mbox{and} \ \lim_{\tau \to 0} p_\tau(\lambda) = 1 $$ for every $\lambda.$
	\end{proof}

The above statement implies the following corollary.
		
		\begin{cor} {\bf (Uppersemicontinuity of the spectrum)}\label{qt1} Let $\varepsilon > 0$ given. Then, for each $z^0 \in \sigma_p(\A)$ there exists $\delta = \delta_\varepsilon > 0$ and $\tau < \delta$  such that the set $$\{z^\tau \in \sigma(\A^\tau); |z^\tau - z^0| < \varepsilon\}$$ is nonempty.
	\end{cor}

\begin{rmk}
	The goal of Propositition \ref{spec1} is to localize the vertical asymptote in the spectrum explicitly and support the later claim that, as $\tau$ vanishes, it becomes arbitrarily far from the imaginary axis. Notice that the proofs of Lemma \ref{t1} and part (c) of Proposition \ref{spec1} (see page \pageref{ee} [Section \ref{secc}])
	have some similarities of algebraic manipulation but have different meaning. The proof of part (c) of Proposition \ref{spec1} takes advantage of the known single-point continuous spectrum to conclude that for $n$ very large the third degree polynomial must have no more then one real root and then quantify the imaginary and real parts of the complex roots. However, the proof of Lemma \ref{t1} makes use of the quadratic structure of the point spectrum of $\A$ in order to conclude the expected approximation.
\end{rmk}

 \begin{figure}[H]	
	\begin{center}
		\includegraphics[scale=.2]{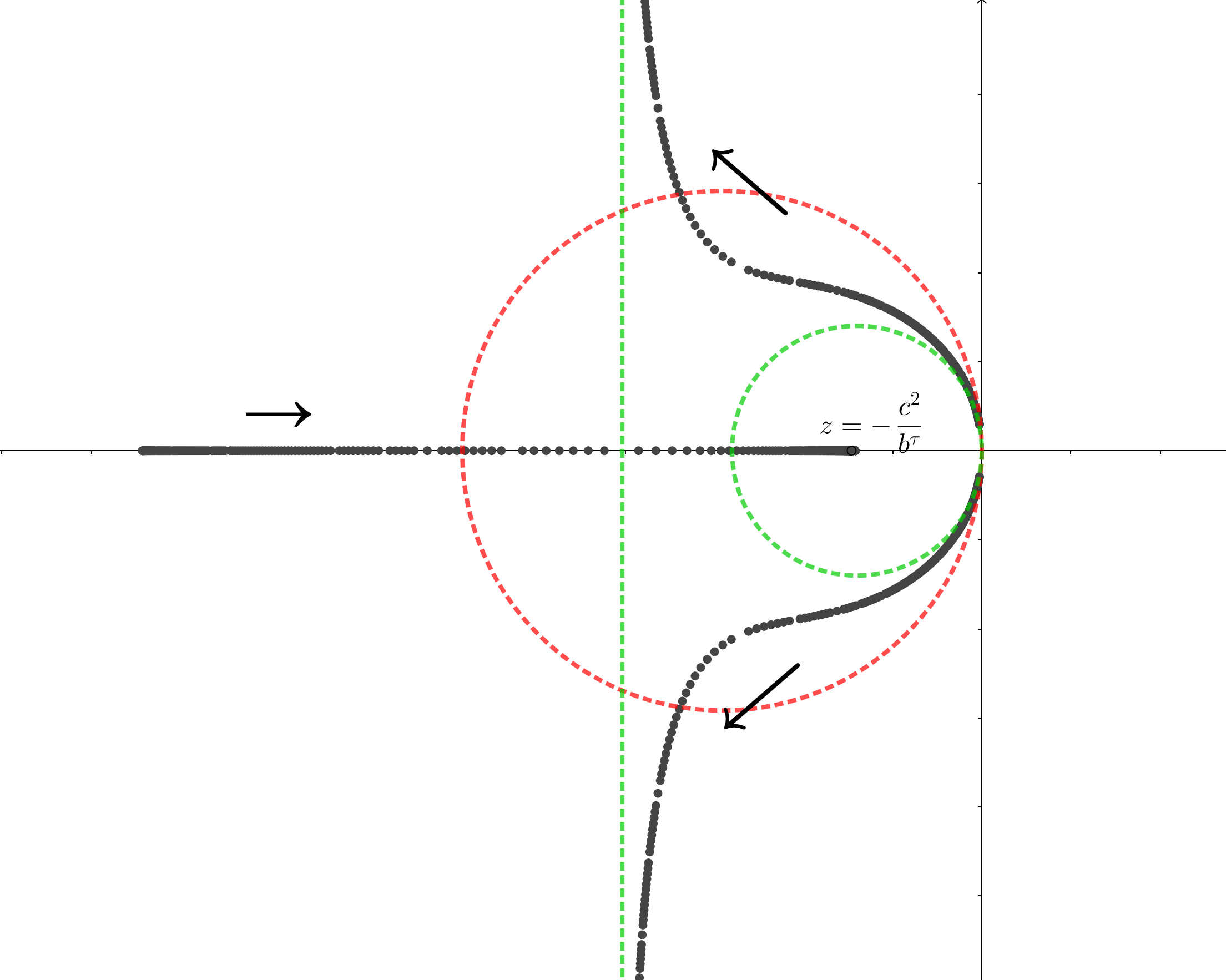}
			\caption{{\bf Graphical representation of $\sigma_p(\A^\tau)$ and $\sigma_c(\A^\tau)$ for a $\tau = \tau_0 \in (0,1]$ fixed}: the circle in red is the same as in Figure \ref{fig:spec_sorder} while the circle in green is centered at $\left(-\frac{c^2}{b^\tau},0\right)$ and has radius $r =  \frac{c^2}{b^\tau}.$ The shape of the eigenvalues is represented by the curve of black dots. The green vertical line is given by $x = -\frac{\gamma^\tau}{2\tau}$ (see Theorem \ref{spec1}). The asymptotic behavior of the eigenvalues is exactly as we described in Theorem \ref{spec1}: as $n$ becomes large, two branches of the eigenvalues have their respective real parts converging to $-\frac{\gamma^\tau}{2\tau}$ while their imaginary parts split into $\pm \infty$, and the other branch converges to the continuous spectrum, which in this case is given by the single point $z = -\frac{c^2}{b^\tau}.$}
		
		\label{fig:third_fix}
	\end{center}
\end{figure}

 \begin{figure}[H]
	\begin{center}
{\includegraphics[scale=.2]{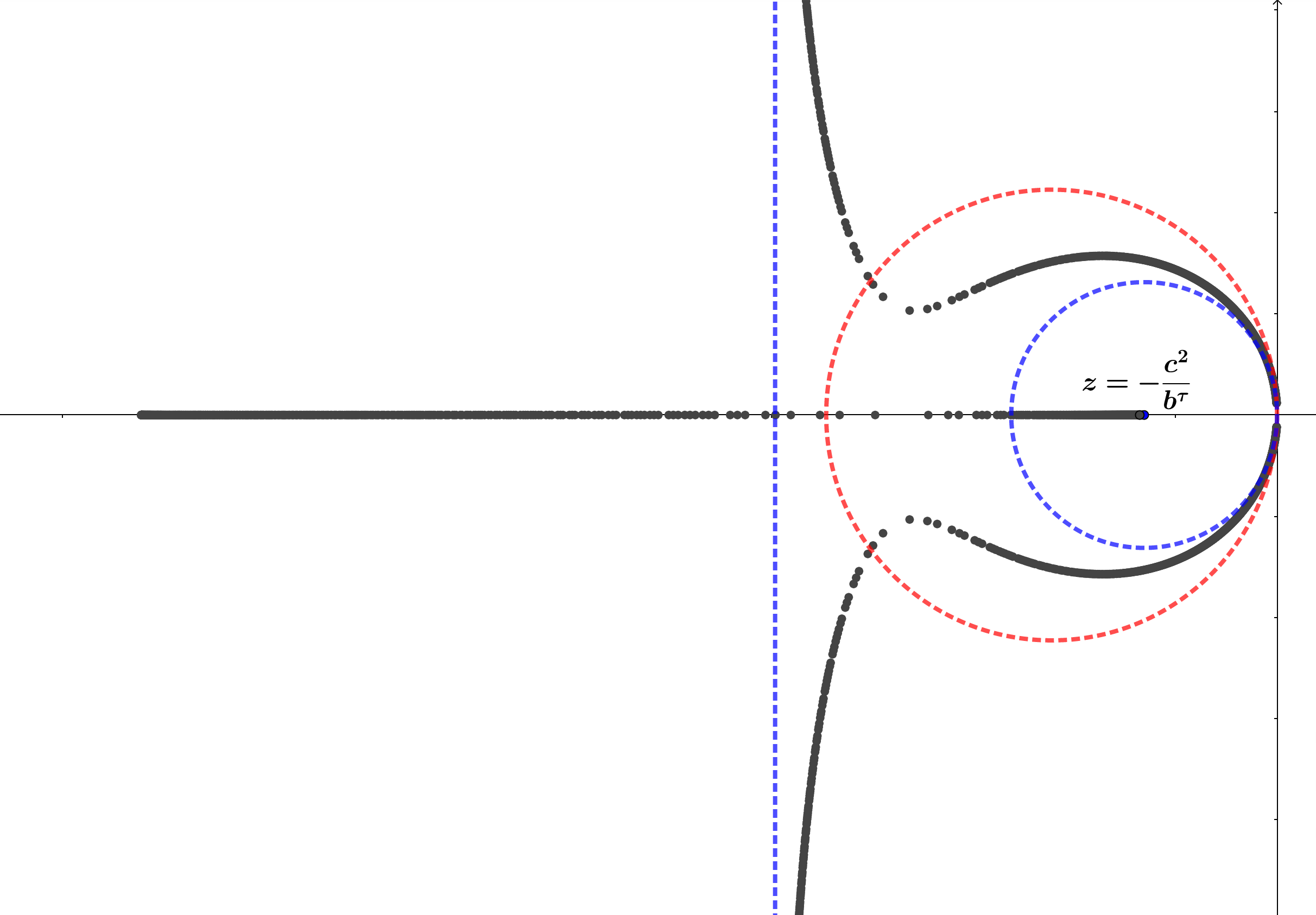}}
	\caption{{\bf Graphical representation of $\sigma_p(\A^\tau)$ and $\sigma_c(\A^\tau)$ for a $\tau = \tau_0 / 10.$}}
	\label{fig:third_meio}
\end{center}
\end{figure}

 \begin{figure}[H]
	\begin{center}
{\includegraphics[scale=.2]{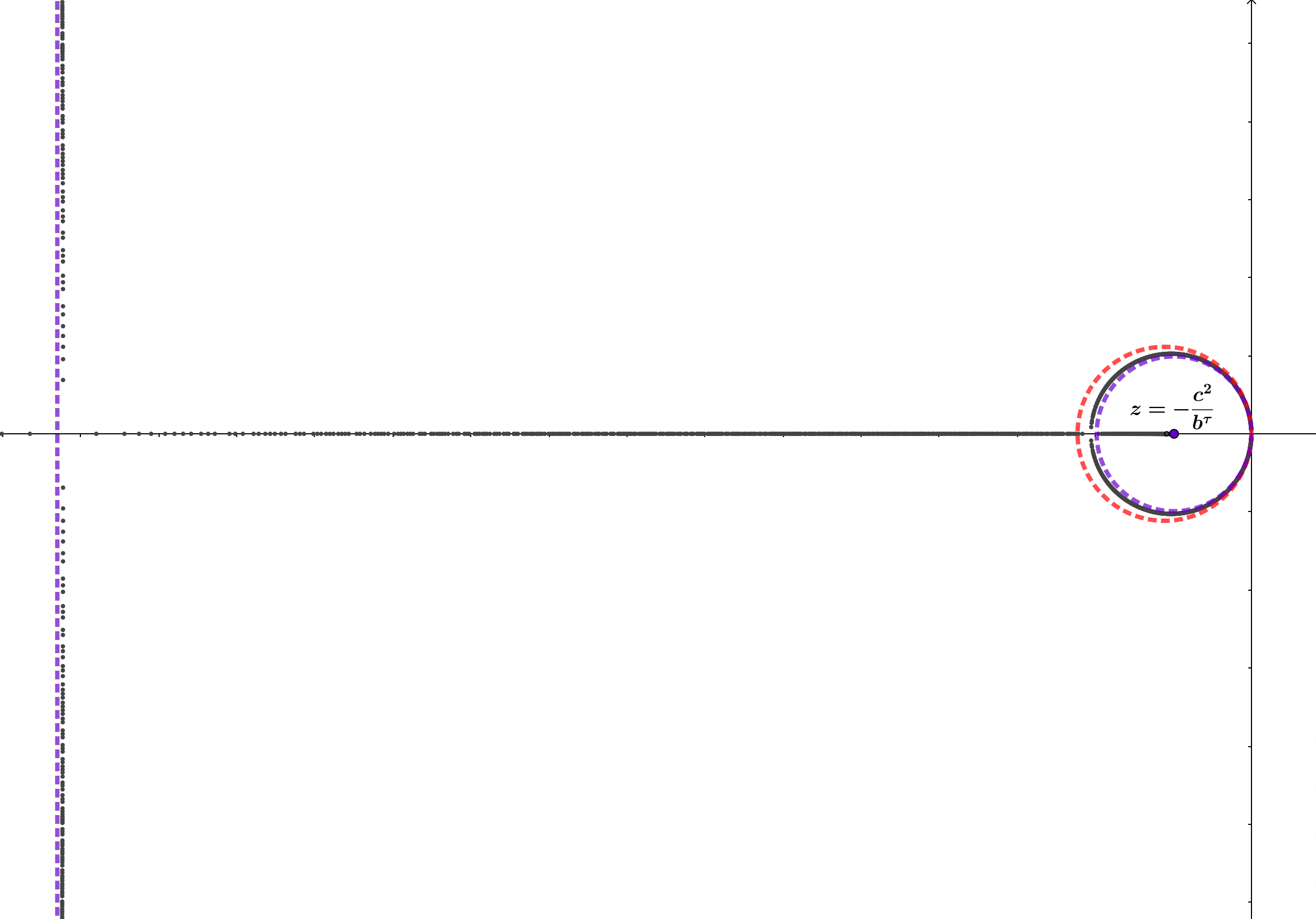}}
	\caption{{\bf Graphical representation of $\sigma_p(\A^\tau)$ and $\sigma_c(\A^\tau)$ for a $\tau=\tau_0/100:$} It is important to observe how the ``vertical" spectrum escapes to ~$-\infty$ as $\tau$ becomes small.}
	\label{fig:third_longe}
%\noindent The remaining part of the paper is devoted to the proofs of the main results. 
	
	%\label{fig:third_longe}
%\end{figure}

\end{center}
\end{figure}	\noindent The remaining part of the paper is devoted to the proofs of the main results.
\section{Proofs}

\subsection{Proof of Theorem \ref{wp-apb-he}}

\noindent  {\bf Part I: Equi-boundedness of the groups}

Let $u^\tau$ be the solution for \eqref{MGTtau} and consider the energy functional $E^\tau(t)$ defined as $$E^\tau(t) = E_0^\tau(t) + E_1^\tau(t),$$ where  \begin{equation}
\label{ef0} E_0^\tau(t) = \dfrac{\alpha}{2}\|u^\tau_t(t)\|_2^2+ \dfrac{c^2}{2}\|\p u^\tau(t)\|_2^2
\end{equation} and \begin{equation}\label{ef1}
E_1^\tau(t) = \dfrac{b^\tau}{2}\left\|\p\left(u^\tau_t(t)+\dfrac{c^2}{b^\tau}u^\tau(t)\right)\right\|_2^2 + \dfrac{\tau}{2}\left\|u^\tau_{tt}(t)+\dfrac{c^2}{b^\tau}u^\tau_t(t)\right\|_2^2 + \dfrac{c^2\gamma^\tau}{2b^\tau}\|u^\tau_t(t)\|_2^2.
\end{equation}
The following differential  identity can be derived for the functional $E_1^{\tau}(t) $ :
\begin{equation}\label{en}
\dfrac{d}{dt}E_1^\tau(t) + \gamma^\tau \|u^\tau_{tt}(t)\|_2^2 = 0,
\end{equation}
where $\gamma^\tau=\alpha - c^2\tau (b^\tau)^{-1} $.
The proof of (\ref{en}) follows  along the same lines as in Lemma 3.1  \cite{jmgt} . The equality is derived first for smooth solution and then extended by density to the  ``energy'' level solutions. For readers convenience the details of the derivation are given in the Appendix. 
%By Theorem 1.3 (\cite{mgtp1}) the energy functional above 
%(which is equivalent to the $\|(u^\tau, u^\tau_t, u^\tau_{tt})\|$ for all $t \geqslant 0$ according to \cite{mgtp1}, Remark 4.1) 
Moreover, for each fixed  value of $\tau > 0,$ the authors in  \cite{jmgt} establish   exponential decay (with decay rates depending on $\tau$)   provided $\gamma^\tau 
 > 0$.  

We aim to prove that the family of semigroups $\mathcal{F}$ is equi-bounded in $\tau.$  In other words, there exists $\tau_0$ small enough such that if we consider $\tau \in (0,\tau_0],$ we can provide an uniform (in $\tau$) bound for the norm of the solutions in $(\Ha,\|\cdot\|\Han) = \Ha^{\tau}$. To achieve this, we  establish  first  topological equivalence 
of energy function with respect to the topology defined on $\Ha^{\tau} $. This is given in the lemma to follow. 

%The equi-boundedness is a consequence of the following lemma.

\begin{lemma}\label{s1}
	Let $U_0 \in \mathbb{H}_0$ and define $\tau_0 \equiv \inf\{c >0; \gamma^\tau > 0 \ \mbox{for all} \ \tau \in (0,c]\}>0$. Then there exist $k = k(\tau_0), K=K(\tau_0)>0$ such that 
\begin{equation}\label{equivine}k\|T^\tau(t)U_0\|^2\Han \leqslant E^\tau(t) \leqslant K\|T^\tau(t)U_0\|^2\Han, \ \tau \in \left(0,\tau_0\right]\end{equation} for all $t \geqslant 0.$ 	
%{\bf \color{magenta} here we have $$k = k(\tau_0) = \min\left\{\dfrac{c^2}{4};\dfrac{\delta^2}{(4+\tau_0)c^2 + 2\delta} ; \dfrac{\delta}{2\delta + 2(2+\tau_0)c^2}\right\}$$ and $$K= K(\tau_0) =\max\left\{\dfrac{c^2}{2}\left(2 + \dfrac{c^2}{\delta}\right);\dfrac{c^2}{2\delta}\left(\!C^\ast\!\!\left(\tau_0+\dfrac{c^2\tau_0}{\delta}+\gamma^{\tau_0}\right)+\dfrac{(\alpha C^\ast+b^{\tau_0}) b^{\tau_0}}{c^2} + b^{\tau_0}\right);\dfrac{1}{2}+\dfrac{c^2}{2\delta}\right\}.$$}
\end{lemma}

\begin{proof}
	We begin with  the second  inequality-as an easier one. In order to get it, we observe that from \eqref{ef1} we have\footnote{we have omitted the obvious dependence on $t.$}:{\small \begin{align*}
	\tiny &E^\tau(t) = \dfrac{b^\tau}{2}\left\|\p\left(u^\tau_t+\dfrac{c^2}{b^\tau}u^\tau\right)\right\|_2^2 + \dfrac{\tau}{2}\left\|u^\tau_{tt}+\dfrac{c^2}{b^\tau}u^\tau_t\right\|_2^2 + \left(\dfrac{c^2\gamma^\tau}{2b^\tau} + \frac{\alpha}{2}\right)\|u^\tau_t\|_2^2 + \dfrac{c^2}{2}\|\p u^\tau\|_2^2 \\ &\leqslant \left(\dfrac{b^\tau+ c^2}{2}\right)\|\p u_t^\tau\|_2^2 + \dfrac{c^2}{2}\left(2 + \dfrac{c^2}{b^\tau}\right)\|\p u^\tau\|_2^2 + \dfrac{c^2}{2b^\tau}\left(\tau+\dfrac{c^2\tau}{b^\tau}+\dfrac{\alpha b^\tau}{c^2} + \gamma^\tau\right)\|u_t^\tau\|_2^2+\dfrac{\tau}{2}\left(1+\dfrac{c^2}{b^\tau}\right)\|u_{tt}^\tau\|_2^2 \\ & \leqslant \dfrac{c^2}{2}\left(2 + \dfrac{c^2}{b^\tau}\right)\|\p u^\tau\|_2^2 +\dfrac{c^2}{2b^\tau}\left[C^\ast\left(\tau+\dfrac{c^2\tau}{b^\tau}+\gamma^\tau\right)+\dfrac{(\alpha C^\ast+b^\tau) b^\tau}{c^2} + b^\tau\right]\|\p u_t^\tau\|_2^2+\dfrac{\tau}{2}\left(1+\dfrac{c^2}{b^\tau}\right)\|u_{tt}^\tau\|_2^2 \\ & \leqslant \dfrac{c^2}{2}\left(2 + \dfrac{c^2}{\delta}\right)\|\p u^\tau\|_2^2 +\dfrac{c^2}{2\delta}\left[\!C^\ast\!\!\left(\tau_0+\dfrac{c^2\tau_0}{\delta}+\gamma^{\tau_0}\right)+\dfrac{(\alpha C^\ast+\delta) \delta}{c^2} + \delta\right]\|\p u_t^\tau\|_2^2+\dfrac{\tau}{2}\left(1+\dfrac{c^2}{\delta}\right)\|u_{tt}^\tau\|_2^2 \\ &\leqslant \max\left\{\dfrac{c^2}{2}\left(2 + \dfrac{c^2}{\delta}\right);\dfrac{c^2}{2\delta}\left[\!C^\ast\!\!\left(\tau_0+\dfrac{c^2\tau_0}{\delta}+\gamma^{\tau_0}\right)+\dfrac{(\alpha C^\ast+\delta) \delta}{c^2} + \delta\right];\dfrac{1}{2}+\dfrac{c^2}{2\delta}\right\} \|U^\tau(t)\|^2_{\tau,0},
	\end{align*}} 
\footnote{we have omitted the obvious dependence on $t.$}
where $C^\ast = C^\ast(n,\Omega)$ is the constant that appears in Poincaré's Inequality.\\
	
Now  the second inequality in (\ref{equivine})  follows after we define
	\begin{equation}\label{bk}
	K(\tau_0)  \equiv  \max\left\{\dfrac{c^2}{2}\left(2 + \dfrac{c^2}{\delta}\right);\dfrac{c^2}{2\delta}\left[\!C^\ast\!\!\left(\tau_0+\dfrac{c^2\tau_0}{\delta}+\gamma^{\tau_0}\right)+\dfrac{(\alpha C^\ast+\delta) \delta}{c^2} + \delta\right];\dfrac{1}{2}+\dfrac{c^2}{2\delta}\right\}.
	\end{equation}

	For the first inequality, fix $\varepsilon > 0$ to be determined later. Peter-Paul inequality then implies that \begin{equation}
	-\frac{c^2\varepsilon\|\p u^{\tau}_t\|_2^2}{2}-\frac{c^2\|\p u^{\tau}\|_2^2}{2\varepsilon} \leqslant c^2(\p u^{\tau},\p u^{\tau}_t) \ \ \mbox{and} \ -\frac{\varepsilon\tau c^2}{2b^\tau}\|u^{\tau}_{tt}\|_2^2-\frac{\tau c^2}{2\varepsilon b^\tau}\|u^{\tau}_t\|_2^2 \leqslant \frac{\tau c^2}{b^\tau}(u^{\tau}_{tt},u^{\tau}_t). \label{epep}
	\end{equation} Then we have {\small \begin{align*}\label{est1}
	&E^\tau(t) = \dfrac{b^\tau}{2}\left\|\p\left(u^\tau_t+\dfrac{c^2}{b^\tau}u^\tau\right)\right\|_2^2 + \dfrac{\tau}{2}\left\|u^\tau_{tt}+\dfrac{c^2}{b^\tau}u^\tau_t\right\|_2^2 + \left(\dfrac{c^2\gamma^\tau}{2b^\tau} + \frac{\alpha}{2}\right)\|u^\tau_t\|_2^2 + \dfrac{c^2}{2}\|\p u^\tau\|_2^2 
	\nonumber \\ &= \dfrac{c^2}{2}\!\left(\!1\!+\!\dfrac{c^2}{b^\tau}\right)\!\!\|\p u^\tau\|_2^2 + \dfrac{b^\tau}{2}\|\p u^\tau_t\|_2^2 +c^2(\p u^\tau, \p u_t^\tau) + \dfrac{\tau}{2}\|u_{tt}^\tau\|_2^2 + \dfrac{\alpha }{2}\!\!\left(1\!+\! \frac{ c^2}{ b^\tau}\!\right)\!\!\|u^\tau_t\|_2^2 + \dfrac{\tau c^2}{b^\tau} (u^\tau_{tt},u^\tau_t) 
	\nonumber	\\ & \geqslant \dfrac{c^2}{2}\!\left(\!1\!+\!\dfrac{c^2}{b^\tau}-\dfrac{1}{\varepsilon}\right)\!\!\|\p u^\tau\|_2^2 + \left(\dfrac{b^\tau}{2}-\frac{c^2\varepsilon}{2}\right)\|\p u^\tau_t\|_2^2 + \dfrac{\tau}{2}\left(1-\dfrac{c^2\varepsilon}{b^\tau}\right)\|u_{tt}^\tau\|_2^2 + \left(\dfrac{\alpha }{2}\!\!\left(1\!+\! \frac{ c^2}{ b^\tau}\!\right)-\dfrac{\tau c^2}{2\varepsilon b^\tau}\right)\|u^\tau_t\|_2^2,
	\end{align*}} 
	Pick an $\varepsilon > 0$ such that $\dfrac{b^\tau}{b^\tau + c^2} < \varepsilon < \dfrac{b^\tau}{c^2}$ and observe that\\
			$$\text{$\dfrac{b^\tau}{2}-\dfrac{c^2\varepsilon}{2}>0$ and $\dfrac{c^2}{2}\bigg(\dfrac{c^2}{2}_1\bigg)-\dfrac{c^2}{2\varepsilon}=\dfrac{c^2}{2}\bigg(\dfrac{c^2}{b^\tau}+1-\dfrac{1}{\varepsilon}\bigg)>0$}$$
			and similarly
			$$\text{$\dfrac{\tau}{2}-\dfrac{\varepsilon\tau c^2}{2b^\tau}>0$ and $\dfrac{\alpha}{2}\bigg(\dfrac{c^2}{b^\tau}+1\bigg)-\dfrac{\tau c^2}{2\varepsilon b^\tau}=\dfrac{\alpha}{2}\bigg(\dfrac{c^2}{b^\tau}+1-\dfrac{1}{\varepsilon}\bigg)=\dfrac{\gamma^\tau}{2}\bigg(\dfrac{c^2}{b^\tau}+1\bigg)+\dfrac{\tau c^2}{2b^\tau}\bigg(\dfrac{c^2}{b^\tau}+1-\dfrac{1}{\varepsilon}\bigg)>0$}.$$
	Hence, picking $$\varepsilon = \dfrac{2b^\tau}{b^\tau + 2c^2}$$ and continuing the lower bound  estimate of $E^\tau(t)$ we have \begin{align*}
	E^\tau(t) &\geqslant \dfrac{c^2}{2}\!\left(\!1\!+\!\dfrac{c^2}{b^\tau}-\dfrac{1}{\varepsilon}\right)\!\!\|\p u^\tau\|_2^2 + \left(\dfrac{b^\tau}{2}-\frac{c^2\varepsilon}{2}\right)\|\p u^\tau_t\|_2^2 + \dfrac{\tau}{2}\left(1-\dfrac{c^2\varepsilon}{b^\tau}\right)\|u_{tt}^\tau\|_2^2 
	\\ &\geqslant \dfrac{c^2}{4}\|\p u^\tau\|_2^2 + \dfrac{\delta^2}{(4+\tau_0)c^2 + 2\delta}\|\p u^\tau_t\|_2^2 + \dfrac{\tau\delta}{2\delta + (4+\tau_0)c^2}\|u_{tt}^\tau\|_2^2 \\ & \geqslant \min\left\{\dfrac{c^2}{4};\dfrac{\delta^2}{(4+\tau_0)c^2 + 2\delta} ; \dfrac{\delta}{2\delta + (4+\tau_0)c^2}\right\} \|U^\tau(t)\|^2_{\tau,0}.
	\end{align*}
	where we have used  $$\dfrac{c^2}{2}\left(1+\dfrac{c^2}{b^\tau}-\dfrac{1}{\varepsilon}\right) = \dfrac{c^2}{4},$$ 
	 \begin{align*}
	\dfrac{b^\tau}{2}-\frac{c^2\varepsilon}{2} = \dfrac{(b^\tau)^2}{4c^2 + 2b^\tau} \geqslant \dfrac{\delta^2}{(4+\tau_0)c^2 + 2\delta}
	\end{align*} and similarly $$\dfrac{\tau}{2}\left(1-\dfrac{c^2\varepsilon}{b^\tau}\right) = \dfrac{\tau b^\tau}{2b^\tau + 4c^2} \geqslant \dfrac{\tau\delta}{2\delta + (4+\tau_0)c^2} .$$
	Setting
	\begin{equation}\label{sk}
	k(\tau_0) \equiv \min\left\{\dfrac{c^2}{4};\dfrac{\delta^2}{(4+\tau_0)c^2 + 2\delta} ; \dfrac{\delta}{2\delta + (4+\tau_0)c^2}\right\}
	\end{equation}
	gives the first part of the inequality in Lemma \ref{s1}. This completes the proof of the Lemma. 

\end{proof}
Lemma \ref{s1} along with the identity (\ref{en}) imply the equi-boundedness of the family $\cF $. Indeed,
 \begin{lemma}\label{equi}Let $U_0 \in \mathbb{H}_0.$ For given $k$ in \eqref{sk} and $K$ in \eqref{bk}, there exists a constant $L_1>0$ (independent on $\tau \in (0, \tau_0) $) such that
 $$ \|T^\tau(t)U_0\|^2\Han \leqslant \dfrac{1}{k} E^\tau(t) \leqslant \dfrac{L_1K}{k}\|U_0\|^2\Han.$$ 
 \end{lemma}
\begin{proof}
	We work with  sufficiently smooth solutions guaranteed by the well-posedness-regularity theory.  The final estimates are obtained via density.
	
	Taking the $L^2-$inner product of \eqref{MGTtau} with $u^\tau_t$ we obtain \begin{equation}\label{aux1}
	b^\tau\|\p u^\tau_t(t)\|_2^2 = \tau \|u^\tau_{tt}(t)\|_2^2 - \dfrac{d}{dt} \left[\dfrac{\alpha}{2}\|u^\tau_t\|_2^2 + \dfrac{c^2}{2}\|\p u^\tau\|_2^2 + \tau(u^\tau_{tt},u^\tau_t)\right].\end{equation}
	
	Multiplying \eqref{aux1} by $\gamma^\tau$ with using \eqref{ef0} gives
	% (which proof can be found in \cite{mgtp1})
	\begin{equation}\label{ut2}
	\gamma^\tau b^\tau\|\p u^\tau_t(t)\|_2^2 = \gamma^\tau\tau \|u^\tau_{tt}(t)\|_2^2 - \gamma^\tau\dfrac{d}{dt}E^\tau_0(t)- \gamma^\tau\tau\dfrac{d}{dt}(u^\tau_{tt},u^\tau_t)
	\end{equation}
	
	Combine \eqref{ut2} with the identity \eqref{en} we obtain
	 \begin{equation}\label{aux2}
	\dfrac{d}{dt}E_1^\tau(t) + \gamma^\tau\dfrac{d}{dt} E^\tau_0(t) + b^\tau\gamma^\tau\|\p u^\tau_t(t)\|_2^2 = \gamma^\tau(\tau-1)\|u^\tau_{tt}(t)\|^2_2 -\gamma^\tau\tau\dfrac{d}{dt}(u^\tau_{tt},u^\tau_t).\end{equation} 
	Since $\tau$ is very small and we assume $0<\tau<1,$ we have $\gamma^\tau(\tau-1)\|u^\tau_{tt}(t)\|^2_2<0$ for all $t\in[0,T].$
	
	Then integrating w.r.t. time from $0$ to $t$ we have \begin{equation}\label{aux3}
	E_1^\tau(t) + \gamma^\tau E^\tau_0(t) + \gamma^\tau b^\tau\int_0^t \|\p u^\tau_t(s)\|_2^2ds\leqslant E_1^\tau(0) + \gamma^\tau E^\tau_0(0)+\gamma^\tau\tau(u^\tau_{tt},u^\tau_t)\big\rvert_0^t.\end{equation} 
	
	From \eqref{en}, we have $E_1^\tau(t)\leqslant E_1^\tau(0)$ and notice that \begin{align*}\tau(u^\tau_{tt},u^\tau_t)\big\rvert_0^t &= \tau(u^\tau_{tt}(t),u^\tau_t(t)) - \tau(u_2,u_1) \\ &\leqslant \dfrac{\tau}{2}\|u^\tau_{tt}(t)\|_2^2 +\dfrac{\tau}{2}\|u_t^\tau(t)\|_2^2+\dfrac{\tau}{2}\|u_2\|_2^2 +\dfrac{\tau}{2}\|u_1\|_2^2 \\ 
	&\leqslant \bigg(1+\dfrac{\tau b}{\alpha c^2}\bigg)\bigg[E_1^\tau(t)+E_1^\tau(0)\bigg]\leqslant2\bigg(\dfrac{\alpha c^2+\tau b}{\alpha c^2}\bigg)E_1^\tau(0)
	.\end{align*} 
	Then
	$$E_1^\tau(t) + \gamma^\tau E^\tau_0(t) + \gamma^\tau b^\tau\int_0^t \|\p u^\tau_t(s)\|_2^2ds\leqslant E_1^\tau(0) + \gamma^\tau E^\tau_0(0)+2\gamma^\tau\bigg(\dfrac{\alpha c^2+\tau b^\tau}{\alpha c^2}\bigg)E_1^\tau(0).$$
	$$E_1^\tau(t) + \gamma^\tau E^\tau_0(t) \leqslant\max\biggl\{\dfrac{\alpha c^2+2\gamma^\tau(\alpha c^2+\tau b^\tau)}{\alpha c^2}, \gamma^\tau\biggr\}E^\tau(0)$$
	Therefore
	\begin{equation}\label{bdd}
	E^\tau(t)\leqslant\dfrac{\max\biggl\{\dfrac{\alpha c^2+2\gamma^{\tau_0}\alpha c^2+\tau b^\tau}{\alpha c^2}, \gamma^{\tau_0}\biggr\}}{\min\{1,\gamma^{\tau_0}\}}E^\tau(0)=L_1E^\tau(0)
	\end{equation}
	
This means that the \textit{total} energy $E^\tau(t)$ is bounded in time by the initial \textit{total} energy.
Thus the proof is obtained.
\end{proof}
 From Lemma \ref{equi} we conclude 
 \begin{clr}
 
There exists a constant $ M > 0$ (independent on $\tau \in (0, \tau_0) $) such that 
 
 \begin{equation}\label{hh1}\|T^{\tau}(t)\|_{\mathcal{L}(\mathbb{H}^{\tau}_0)} \leqslant  M ,  \forall t > 0. \end{equation}
 \end{clr}

\noindent  {\bf Part II: Uniform (in $\tau$) decay rates}

In order to prove the uniformity of the decay rates we use the Pazy-Datko Theorem. The first step consists of showing the the map $t \mapsto \|T^\tau(t)U_0\|^2_{\tau,0}$ belongs to $L^1(0,\infty;\mathbb{H}_0).$ This is the statement of the next Lemma.

\begin{lemma} There exists $\overline{K}>0$ independent on $\tau$ such that 
	\begin{equation}\label{Pz-Dt}
	\int_0^\infty \|T^\tau(s)U_0\|^2\Han ds \leqslant \overline{K}\|U_0\|^2\Han < \infty.
	\end{equation}
\end{lemma}
\begin{proof}
	Multiplying the identity \eqref{en} by $2\tau$ gives
\begin{equation}\label{dr1}
2\tau\dfrac{d}{dt}E_1^\tau(t) + 2\tau\gamma^\tau \|u^\tau_{tt}(t)\|_2^2 = 0
\end{equation}
	
	Multiply \eqref{aux1} by $\gamma^\tau$ we have
\begin{equation}\label{dr2}
\gamma^\tau b^\tau\|\p u^\tau_t(t)\|_2^2 = \gamma^\tau\tau \|u^\tau_{tt}(t)\|_2^2 - \gamma^\tau\dfrac{d}{dt}E^\tau_0(t)+ \gamma^\tau\tau\dfrac{d}{dt}(u^\tau_{tt},u^\tau_t)
\end{equation}
	% (which proof can be found in \cite{mgtp1})
	With \eqref{dr1} and \eqref{dr2} we get \begin{equation}\label{aux2}
	2\tau\dfrac{d}{dt}E_1^\tau(t) + \gamma^\tau\dfrac{d}{dt} E^\tau_0(t) + \tau\gamma^\tau\|u^\tau_{tt}(t)\|^2_2 + b^\tau\gamma^\tau\|\p u^\tau_t(t)\|_2^2 =  -\gamma^\tau\tau\dfrac{d}{dt}(u^\tau_{tt},u^\tau_t).\end{equation} Then integrating w.r.t. time from $0$ to $t$ we have \begin{equation}\label{aux3}
	2\tau E_1^\tau(t) + \gamma^\tau E^\tau_0(t) + \gamma^\tau\int_0^t \left[\tau\|u^\tau_{tt}(s)\|^2_2 + b^\tau\|\p u^\tau_t(s)\|_2^2\right]ds =  2\tau E_1^\tau(0) + \gamma^\tau E^\tau_0(0)-\gamma^\tau\tau(u^\tau_{tt},u^\tau_t)\big\rvert_0^t.\end{equation}
	Notice that from Lemma \ref{equi} we have
	$$\tau(u^\tau_{tt},u^\tau_t)\big\rvert_0^t\leqslant (1+\tau_0)\left(\|T^\tau(t)U_0\|_\tau^2 + \|U_0\|_\tau^2\right) \leqslant \dfrac{(1+\tau_0)(L_1K+k)}{k}\|U_0\|_\tau^2.$$
	Therefore, \begin{align}
	\label{aux4} \int_0^t \left[\tau\|u^\tau_{tt}(s)\|_2^2 + b^\tau\|\p u^\tau_t(s)\|_2^2\right]ds \nonumber &\leqslant \dfrac{\gamma^\tau(1+\tau_0)(L_1K+k)+(2\tau_0+\gamma^\tau)L_1Kk}{k\gamma^\tau}\|U_0\|^2_{\tau,0} \\ &\leqslant \dfrac{\gamma^{\tau_0}(1+\tau_0)(L_1K+k)+(2\tau_0+\gamma^{\tau_0})L_1Kk}{k\gamma^{\tau_0}}\|U_0\|^2_{\tau,0} = M_1\|U_0\|^2_{\tau,0}.
	\end{align}
	Similarly, taking the $L^2-$inner product of \eqref{MGTtau} with $u^\tau$ we have \begin{equation}
	\label{aux5} \dfrac{b^\tau}{2} \dfrac{d}{dt} \|\p u^\tau\|_2^2 + c^2 \|\p u^\tau(t)\|_2^2 = \alpha\|u^\tau_t(t)\|_2^2 + \dfrac{d}{dt}\left[\dfrac{\tau}{2}\|u^\tau_t|\|_2^2 - \tau(u^\tau_{tt},u^\tau) - \alpha(u^\tau_t,u^\tau)\right]
	\end{equation} and integrating \eqref{aux5} w.r.t time from $0$ to $t$ we have \begin{align}
	\label{aux6} \dfrac{b^\tau}{2} \|\p u^\tau(t)\|_2^2 &+ c^2\int_0^t \|\p u^\tau(s)\|_2^2 ds\nonumber\\
	&=  \dfrac{b^\tau}{2} \|\p u_0\|_2^2+\alpha\int_0^t\|u^\tau_t(s)\|_2^2ds + \left[\dfrac{\tau}{2}\|u^\tau_t|\|_2^2 - \tau(u^\tau_{tt},u^\tau) - \alpha(u^\tau_t,u^\tau)\right]\biggr\rvert_0^t \nonumber \\ & \leqslant b^\tau\|U_0\|^2_{\tau,0} + \dfrac{\alpha M_1}{b^\tau} \|U_0\|^2_{\tau,0} + \dfrac{\big[(\alpha + \tau_0)C^\ast + 1\big](L_1K+k)}{k}\|U_0\|^2_{\tau,0} \nonumber \\ &= \dfrac{kb^\tau + \alpha k M_1 + b^\tau\big[(\alpha + \tau_0)C^\ast + 1\big](L_1K+k)}{b^\tau k}\|U_0\|^2_{\tau,0} \nonumber \\ & \leqslant  \dfrac{k\delta + \alpha k M_1 + \delta\big[(\alpha + \tau_0)C^\ast + 1\big](L_1K+k)}{\delta k}\|U_0\|^2_{\tau,0} = M_2\|U_0\|^2_{\tau,0}.
	\end{align}
	Hence, by \eqref{aux4} and \eqref{aux6} we get \begin{equation}\label{es0}
	\int_0^\infty \|T^\tau(s)U_0\|^2_\tau ds \leqslant \dfrac{M_1 + M_2}{1 + \delta + c^2} \|U_0\|^2_\tau = M_3\|U_0\|^2_\tau < \infty.
	\end{equation}
	
	Therefore, according to Theorem 4.1 (\cite{pazy}, p. 116) the rate $\omega$ can be determined  as follows: We first chose a number $\rho$ such that $0 < \rho < M_3^{-1},$ then we define a number $\eta_0 = M_3\rho^{-1}$ and choose another number $\eta$ such that $\eta > \eta_0.$ The rate is then given by $$\omega = -\dfrac{1}{\eta}\log(M_3\rho) > 0,$$ and is clearly independent on $\tau.$ The proof is thus completed. 
\end{proof}

\subsection{Proof of Theorem \ref{wp-apb-he-n}}

\noindent {\bf (a) Well-posedness} 

The well-posedness follows directly from Theorem 1.4 in \cite{mgtp1} and the fact that on the space $\mathbb{H}_1$ the standard sum norm $\|\cdot\|$ and $\|\cdot\|\Hbn$ are equivalent for each $\tau \in (0,1]$.

\

\noindent {\bf (b) Equi-boundedness  on  $\Hb^{\tau} $ and uniform (in $\tau$) exponential stability}

Let $u^\tau$ be the solution for \eqref{MGTtau} and consider the energy functional $\E^\tau(t)$ defined as \begin{equation}\E^\tau(t) = E^\tau(t) + \|Au^\tau (t)\|_2^2. \label{hef}\end{equation}

By Theorem 1.3 (\cite{mgtp1}) the energy functional above (which is equivalent to the $\|(u^\tau, u^\tau_t, u^\tau_{tt})\|$ for all $t \geqslant 0$) decays exponentially with time [for each fixed $\tau > 0,$ provided $\gamma^\tau \equiv \alpha - c^2\tau (b^\tau)^{-1} > 0$.

As in the previous theorem, the equi-boundedness of the family $\mathcal{F}_1$ follows from the lemma:

\begin{lemma}\label{s2}
	Let $U_0 \in \mathbb{H}_1$ and define $\tau_0 \equiv \inf\{c >0; \gamma^\tau > 0 \ \mbox{for all} \ \tau \in (0,c]\}>0$. Then there exist $k_1=k_1(\tau_0)$ and $K_1=K_1(\tau_0)$ such that $$k_1\|T^\tau(t)U_0\|^2\Hbn \leqslant \E^\tau(t) \leqslant K_1\|T^\tau(t)U_0\|^2\Hbn, \ \tau \in \left(0,\tau_0\right]$$ for all $t \geqslant 0.$
\end{lemma}

The proof of Lemma \ref{s2} as well as the conclusion of the equi-boundedness in $\mathbb{H}_1$  capitalizes on the estimates already derived 
 for $\mathbb{H}_0$. We shall focus on  additional terms which need to be estimated additionally.  
 
%{\color{blue}\bf \small Side notes (this would not go to the paper): The configuration on $\tau$ and $\alpha$ follows from the fact that: \begin{align*}
%	\gamma^\tau > 0 \Leftrightarrow \alpha > \dfrac{c^2\tau}{\delta + \tau c^2}  \Leftrightarrow \tau c^2(\alpha -1) > -\delta \alpha
%	\end{align*} which is true for every $\tau > 0$ in case $\alpha \geqslant 1$ but in case $0 < \alpha < 1$ we need $\tau < \dfrac{\delta\alpha}{c^2(1-\alpha)}.$
%}

%Moreover, the argument for showing the uniformity of the decay rates is also the same as before, as long as we also guarantee that $t \mapsto \|T^\tau(t)U_0\|^2_{\tau,1}$ belongs to $L^1(0,\infty; \mathbb{H}_1)$, but for that all we need (in light of what we have done for $\mathbb{H}_0$) is to show that $t \mapsto \|Au^\tau(t)\|_2^2$ is in $L^1.$

%But notice that taking the $L^2-$inner product of \eqref{MGTtau} with $Au^\tau$ and integrating w.r.t time from $0$ to $t$ we have
%	\begin{align}
%	\label{aux7} \dfrac{b^\tau}{2} \|Au^\tau\|_2^2 + c^2\int_0^t \|A u^\tau\|_2^2 ds &=  \dfrac{b^\tau}{2} \|A u_0\|_2^2+\alpha\int_0^t\|\p u^\tau_t\|_2^2ds + \left[\dfrac{\tau}{2}\|\p u^\tau_t|\|_2^2 - \tau(u^\tau_{tt},Au^\tau) - \alpha(u^\tau_t,Au^\tau)\right]\biggr\rvert_0^t \end{align} and since we have control of all the terms on the right hand side, the proof is complete.\\
 Lemma \ref{s2} is obtained from   Lemma \ref{s1} and the  estimates  already derived for the space $\mathbb{H}_0^{\tau}$  by adding the term $||A u^{\tau} (t)||^2_2 $.  This gives the inequality stated in Lemma \ref{s2}. 
	Recall \eqref{bdd}, we will   obtain the apriori bound for $\E^{\tau}(t) $ from the relation 
	$$\E^{\tau}(t) = E^{\tau}(t) + ||Au^{\tau}(t)||^2_2 \leq L_1E^{\tau}(0) + \sup_{t \geq 0 } ||Au^{\tau}(t)||^2_2.$$
	To achieve the goal, the second term needs to be accounted for. 
	To estimate the second term we employ the equality by taking the $L^2-$inner product of \eqref{MGTtau} with the multiplier $Au^\tau$ and integrating w.r.t time from $0$ to $t$
		\begin{align}
	\label{aux7} \dfrac{b^\tau}{2} \|Au^\tau(t)\|_2^2 &+ c^2\int_0^t \|A u^\tau(s)\|_2^2 ds =  \dfrac{b^\tau}{2} \|A u_0\|_2^2\nonumber\\
	&+\alpha\int_0^t\|\p u^\tau_t(s)\|_2^2ds + \left[\dfrac{\tau}{2}\|\p u^\tau_t|\|_2^2 - \tau(u^\tau_{tt},Au^\tau) - \alpha(u^\tau_t,Au^\tau)\right]\biggr\rvert_0^t \end{align}
	and by using  the already obtained estimates in $\Ha^{\tau}$
\begin{align}
\label{aux8}\dfrac{(b^\tau-\varepsilon)}{2} \|Au^\tau(t)\|_2^2 + c^2\int_0^t \|A u^\tau(s)\|_2^2 ds&\leqslant b^\tau\|U_0\|^2\Hbn + \alpha M3\|U_0\|^2\Hbn+\dfrac{\big(\dfrac{\tau}{2}+\tau C_\varepsilon+\alpha\big)(K_1+k_1)}{k_1}\|U_0\|^2\Hbn\nonumber \\ &=\dfrac{2k_1b^\tau+2k_1 \alpha M_3+\big[\tau(1+2C_\varepsilon)+2\alpha\big](K_1+k_1)}{2k_1}\|U_0\|^2\Hbn \nonumber\\ &\leqslant\dfrac{2k_1\delta+2k_1 \alpha M_3+\big[\tau_0(1+2C_\varepsilon)+2\alpha\big](K_1+k_1)}{2k_1}\|U_0\|^2\Hbn\nonumber \\ &=\hat{C}\|U_0\|^2_{\tau,1}.
\end{align}
Rescaling $\varepsilon$ allows to estimate $ sup_{t} ||A u^{\tau} (t) || $, hence $\E^{\tau}(t) $. 
	  Then we have $$ \|T^\tau(t)U_0\|^2\Hbn \leqslant \dfrac{1}{k_1} \E^\tau(t) \leqslant\bigg(\dfrac{\hat{C}K_1}{k_1}\bigg)\|U_0\|^2\Hbn,$$ from where  it follows that 
	  the groups are  equibounded also on $\Hb^{\tau} $, i.e.,
	  \begin{clr} There exists a constant $ N > 0$ [independent on $\tau \in (0, \tau_0 ] $] such that 
	  	\begin{equation}\label{hh2}\|T^\tau(t)\|_{\mathcal{L}(\mathbb{H}^{\tau}_1)} \leqslant \bigg(\dfrac{\hat{C}K_1}{k_1}\bigg)^{1/2} = N.\end{equation}
	  \end{clr}

		%{\color{blue}\bf \small Side notes (this would not go to the paper): The configuration on $\tau$ and $\alpha$ follows from the fact that: \begin{align*}
		%	\gamma^\tau > 0 \Leftrightarrow \alpha > \dfrac{c^2\tau}{\delta + \tau c^2}  \Leftrightarrow \tau c^2(\alpha -1) > -\delta \alpha
		%	\end{align*} which is true for every $\tau > 0$ in case $\alpha \geqslant 1$ but in case $0 < \alpha < 1$ we need $\tau < \dfrac{\delta\alpha}{c^2(1-\alpha)}.$
		%}
		
		%\
		
	As for exponential uniform decays, we shall evoke again  the Pazy-Datko Theorem.  This  shows the existence of uniform (in $\tau$) decay rate with existence of $\overline{K_1}>0$ independent on $\tau$ such that 
		\begin{lemma}
			\begin{equation}\label{Pz-Dt1}
			\int_0^\infty \|T^\tau(s)U_0\|^2\Hbn ds \leqslant \overline{K_1}\|U_0\|^2\Hbn < \infty.
			\end{equation}
		\end{lemma}
		\begin{proof}
			Direct from \eqref{aux8}, we have
			\begin{align}
			%\label{aux7}
			  c^2\int_0^t \|A u^\tau(s)\|_2^2 ds
			\leqslant N_1\|U_0\|^2\Hbn.
			\end{align}
			Hence, by \eqref{es0} and \eqref{aux7} we have
			\begin{equation}\label{es1}
			\int_0^\infty \|T^\tau(s)U_0\|^2\Hbn ds \leqslant \dfrac{M_3 + N_1}{c^2} \|U_0\|^2\Hbn = N_2\|U_0\|^2\Hbn < \infty.
			\end{equation}
			
			Thus by Theorem 4.1 (\cite{pazy}, p. 116) the rate $\omega$ can be taken as we first chose a number $\rho$ such that $0 < \rho < N_2^{-1},$ then we define a number $\eta_0 = N_2\rho^{-1}$ and choose another number $\eta$ such that $\eta > \eta_0.$ The rate is then given by $$\omega = -\dfrac{1}{\eta}\log(N_2\rho) > 0,$$ and is clearly independent on $\tau.$ Then the proof is completed. 
	\end{proof}

\subsection{Proof of Theorem \ref{cr}}
{\bf Proof of part (a)-convergence rates }. 

Let $x^\tau = u^\tau - u^0$ where $u^\tau$ and $u^0$ are the solutions for the problems \eqref{MGTtau} and \eqref{MGT0} respectively with the same initial values for $u(t=0)$ and $u_t(t=0)$. By  taking the difference of the two problems we can write a $x^\tau-$problem given by \begin{equation}\begin{cases}
\alpha x^\tau_{tt} +c^2A x^\tau + \delta A x^\tau_t = -\tau u^\tau_{ttt} - \tau c^2 Au^\tau_t\ &\mbox{in} \ (0,T) \times \Omega, \\
x^\tau(0) = 0, x^\tau_t(0) = 0.
\end{cases}\label{xp}\end{equation} 

Observe that since $\mathbb{H}_2$ is $\mathbb{H}_0$  subject to the multiplication by $A^{1/2} $ where the latter leaves the dynamics invariant 
%with all the pieces boosted equally in terms of regularity, i.e., the first and second pieces went from $\D{\p}$ to $\D{A}$ and the last went from $L^2(\Omega)$ to $\D{\p}.$ 
and  $\A^\tau$ generates a $C_0-$group in $\mathbb{H}_0,$ we also have $\A^\tau$ generating a $C_0-$group in $\mathbb{H}_2.$ 
%which is moreover equi-bounded for $\tau$ small and exponentially stable with uniform decay rates for $\tau$ small. That is, there exist $\overline{M}>0$ and $\overline{\omega}>0$ such that, for $\tau$ small, $$\|T^\tau(t)\|_{\mathcal{L}(\mathbb{H}_2)} \leqslant \overline{M}e^{-\overline{\omega} t}, \ t \geqslant 0.$$

We aim to prove that $$\|P T^\tau(t) U_0 - T(t)PU_0\|_{\mathbb{H}_0^0 }^2 \leqslant\tau C\|U_0\|_{\mathbb{H}_2}^2$$ which is the same as showing that $$\|\p x^\tau(t)\|_2^2 + \|\p x^\tau_t(t)\|_2^2 \leqslant  \tau C\|U_0\|_{\mathbb{H}_2}^2$$
for all $t\in[0,T].$

{\bf Step 1: Reconstruction of $\|A^{1/2} x^\tau(t)\|_2^2$} 

\medskip

We start by taking the $L^2-$inner product of $x^\tau-$equation \eqref{xp} with $x^\tau_t$. This gives $$
\alpha ( x^\tau_{tt}(t),x^\tau_t(t)) +c^2(A x^\tau(t),x^\tau_t(t)) + \delta(A x^\tau(t)_t,x^\tau_t(t)) = -\tau (u^\tau_{ttt}(t),x^\tau_t(t)) - \tau c^2(Au^\tau_t(t),x^\tau_t(t))$$ which can be rewritten as \begin{equation}\dfrac{\alpha}{2}\dfrac{d}{dt}\| x^\tau_t\|_2^2 +\dfrac{c^2}{2}\dfrac{d}{dt}\|A^{1/2}x^\tau\|_2^2 + \delta\|A^{1/2}x^\tau_t(t)\|_2^2 = -\tau(u^\tau_{ttt}(t),x^\tau_t(t))-\tau c^2(\p u^\tau_t(t),\p x_t^\tau(t)).\label{s21}\end{equation} 

We now integrate \eqref{s21} with respect to time from $0$ to $t \in (0,T]$ . This gives
\begin{align*}
\dfrac{\alpha}{2}\| x^\tau_t(t)\|_2^2 &+\dfrac{c^2}{2}\|A^{1/2}x^\tau(t)\|_2^2 + \delta\int_0^t\|A^{1/2}x^\tau_t(s)\|_2^2ds\nonumber\\
&\leqslant \dfrac{\delta}{2}\int_0^t\|A^{1/2}x^\tau_t(s)\|_2^2ds+\dfrac{\tau C^* }{\delta}\int_{0}^{t}\|\sqrt{\tau}u^\tau_{ttt}(s)\|_2^2ds + \dfrac{\tau^2 c^4T}{\delta}\sup_{t\in[0,T]}\|\p u^\tau_t(t)\|_2^2
\end{align*}
where we have used the zero initial conditions of the $x^\tau-$equation and $C^*$ is the Poincaré's constant.
Then
\begin{align}\label{s22}
\dfrac{\alpha}{2}\| x^\tau_t(t)\|_2^2 +\dfrac{c^2}{2}\|A^{1/2}x^\tau(t)\|_2^2&+ \dfrac{\delta}{2}\int_0^t\|A^{1/2}x^\tau_t(s)\|_2^2ds\nonumber\\
&\leqslant\dfrac{\tau  C^*}{\delta}\int_{0}^{t}\|\sqrt{\tau}u^\tau_{ttt}(s)\|_2^2ds  + \dfrac{\tau^2 c^4T}{\delta}\sup_{t\in[0,T]}\|\p u^\tau_t(t)\|_2^2
\end{align}

This was the reconstruction we needed.

\ 

\noindent {\bf Step 2: Reconstruction of $\|A^{1/2} x^\tau_t(t)\|_2^2$} 

\medskip

We start by taking the $L^2-$inner product of $x^\tau-$equation \eqref{xp} with $Ax^\tau_t$. This gives $$
\alpha ( x^\tau_{tt}(t),Ax^\tau_t(t)) +c^2(A x^\tau(t),Ax^\tau_t(t)) + \delta(A x^\tau_t(t),Ax^\tau_t(t)) = -\tau (u^\tau_{ttt}(t),Ax^\tau_t(t))-\tau c^2(Au^\tau_t(t),Ax^\tau_t(t))$$ which can be rewritten as \begin{equation}\dfrac{\alpha}{2}\dfrac{d}{dt}\|A^{1/2} x^\tau_t\|_2^2+\dfrac{c^2}{2}\dfrac{d}{dt}\|Ax^\tau\|_2^2 + \delta\|Ax^\tau_t(t)\|_2^2 = -\tau(u_{ttt}(t),Ax^\tau_t(t))-\tau c^2(Au^\tau_t(t),Ax^\tau_t(t)).\label{s11}\end{equation} 

We now integrate \eqref{s11} with respect to time from $0$ to $t \in (0,T]$. This gives
\begin{align}
\dfrac{\alpha}{2}\|A^{1/2} x^\tau_t(t)\|_2^2&+\dfrac{c^2}{2}\|Ax^\tau(t)\|_2^2 +\delta\int_0^t \|Ax^\tau_t(s)\|_2^2ds\nonumber\\
&\leqslant \dfrac{\delta}{2}\int_0^t \|Ax^\tau_t(s)\|_2^2ds +\frac{\tau }{\delta}\int_{0}^{t}\|\sqrt{\tau}u^\tau_{ttt}(s)\|_2^2ds + \dfrac{\tau^2 c^4T}{\delta}\sup_{t\in[0,T]}\|A u^\tau_t(t)\|_2^2
\end{align}
Then

 \begin{equation}\dfrac{\alpha}{2}\|A^{1/2} x^\tau_t(t)\|_2^2+\dfrac{c^2}{2}\|Ax^\tau(t)\|_2^2 +\dfrac{\delta}{2}\int_0^t \|Ax^\tau_t(s)\|_2^2ds \leqslant\frac{\tau }{\delta}\int_{0}^{t}\|\sqrt{\tau}u^\tau_{ttt}(s)\|_2^2ds + \dfrac{\tau^2 c^4T}{\delta}\sup_{t\in[0,T]}\|A u^\tau_t(t)\|_2^2\label{s12},\end{equation} where we have used the zero initial conditions of the $x^\tau-$equation.

This was the reconstruction we needed.

\ 

\noindent {\bf Step 3: Uniform (in $\tau$) bound for $\displaystyle\int_{0}^{t}\|\sqrt{\tau}u^\tau_{ttt}(s)\|_2^2ds $} 

\medskip

Recall that the problem \eqref{MGTtau} is written abstractly as (see \eqref{op1}) \begin{equation}\begin{cases}M_\tau U^\tau_t(t) = M_\tau\mathcal{A}^\tau U^\tau(t), \ t>0, \\ U^\tau(0)=U_0=(u_0,u_1,u_2)^T. \end{cases}\label{AbS2}\end{equation} 

In order to estimate $\|\sqrt{\tau}u^\tau_{ttt}(t)\|_2^2,$ we differentiate \eqref{AbS2} in time, which leads us to  \begin{equation}\begin{cases}M_\tau U^\tau_{tt}(t) = M_\tau\mathcal{A}^\tau U^\tau_t(t), \ t>0, \\ U_t^\tau(0)=\A^\tau U_0. \end{cases}\label{AbS3}\end{equation} and by relabeling $V^\tau = U^\tau_t$ we can further rewrite \begin{equation}\begin{cases}M_\tau V^\tau_{t}(t) = M_\tau\mathcal{A}^\tau V^\tau(t), \ t>0, \\ V^\tau(0) = V_0 = \A^\tau U_0. \end{cases}\label{AbS4}\end{equation}

Now, since we are considering $U_0 \in \Hc$, which means $u_0,u_1 \in \D{A}$ and $u_2 \in \D{\p}$, we have $\A^\tau U_0 \in \Hb.$ Therefore, by Theorem \ref{wp-apb-he} and Remark \ref{rmkh2} we get \begin{equation}\label{uttt}\|\sqrt{\tau}u^\tau_{ttt}(t)\|_2^2 \leqslant \|T^\tau(t)V_0\|\Hbn^2 \leqslant \overline{M_1}^2 \|V_0\|\Hbn^2 \leqslant \dfrac{\overline{K}}{\tau} \|U_0\|^2_{\mathbb{H}_2},\end{equation} for all $t\in[0,T],$ where $\overline{K}$ does not depend on $\tau.$ \newpage
We also obtain
\begin{equation}\label{ttt}
\gamma^\tau\tau\int_{0}^{t}\|u^\tau_{ttt}(s)\|_2^2ds \leqslant C\|U_0\|^2_{\mathbb{H}_2},
\end{equation}
where $C$ does not depend on $\tau.$
\begin{proof}
	Apply \eqref{en} to time derivatives. This gives 
	$$\gamma \int_0^T |u^\tau_{ttt}|^2 \leq C | A^{1/2} u^\tau_{tt}(0)|^2 + |A^{1/2} u^\tau_t(0)|^2 + \tau |u^\tau_{ttt}(0) |^2 $$
	From the  equation read of $\tau u^\tau_{ttt}(0) $
	$$ \tau |u^\tau_{ttt}(0)|^2 \leq \frac{C}{\tau} [ |u_2|^2 + |Au_0|^2 + |Au_1|^2 ] $$
	This gives  the conclusion in \eqref{ttt}.
\end{proof}
\noindent {\bf Step 4: Collecting the estimates}

By adding \eqref{s22} and \eqref{s12} and using \eqref{ttt} and remark \ref{rmkh2} we conclude
\begin{align}
\| A^{1/2} x^\tau(t)\|_2^2 + \|A^{1/2}x^\tau_t(t)\|_2^2 \leqslant \tau C\|U_0\|^2_{\mathbb{H}_2}.
\end{align}
 This finishes the proof of part (a) of Theorem \ref{cr}.

{\bf Proof of part (b)-strong convergence}.
This amounts to showing that given $U_0 \in \Ha^{\tau} $ and given $\varepsilon > 0 $  there exist $\delta > 0$ such that if $\tau < \delta$ then $$\text{$\|P T^\tau(t) U_0 - T(t)PU_0\|^2_{\mathbb{H}_0^0} < \varepsilon$} $$
for all times $t>0.$ The strategy is prove that for any given fixed time $T$ the above inequality is true and then to choose a suitable $T$ such that for $t > T$ the energy is still bounded above by $\varepsilon.$

\

{\bf Step 1: Finite time.} Let $U_0 \in \mathbb{H}_0$ and $T >0.$ Let $\varepsilon > 0 $ be arbitrary. 

%{\bf Step 1: Finite time.} Let $U_0 \in \mathbb{H}_0$ and $T >0.$ 

Since $\mathbb{H}_2$ is dense in $\mathbb{H}_0$, if we define $$\varepsilon' =\dfrac{ \varepsilon }{ 2(M+M_0)},$$ we can find $U_{\varepsilon'} \in \mathbb{H}_2$ such that $$\|U_0-U_{\varepsilon'}\|^2_{\tau,0} <\varepsilon'.$$ 

Define $$\delta = \delta_\varepsilon = \dfrac{\varepsilon}{2C\|U_{\varepsilon'}\|^2_{\mathbb{H}_2}}$$ (where $C$ comes from Step 4 in the proof of part (a) considering $U_{\varepsilon'}$ as the initial condition) and notice that $\delta \to 0$ as $\varepsilon \to 0.$ 

Then, for $\tau < \delta$ we estimate by using \eqref{hh1}, \begin{align*} \|P T^{\tau}(t) U_0\!\!- T(t)PU_0\|^2_{P(\mathbb{H}_0)} & \leqslant \|P T^{\tau}(t)(U_0\!\!-\!\!U_{\varepsilon'}) - T(t)P(U_0\!\!-\!\!U_{\varepsilon'})\|^2_{\Ha^0} + \|P T^{\tau}(t) U_{\varepsilon'}\!\!-\!\!T(t)PU_{\varepsilon'}\|^2_{\Ha^0} \\ &\leqslant \|T^{\tau}(t)\|^2_{\mathcal{L}(\mathbb{H}_0^\tau)}\|U_0\!\!-\!\!U_{\varepsilon'}\|^2_{\tau,0}  + \|T(t)\|^2_{\mathcal{L}(\Ha^0)}\|U_0\!\!-\!\!U_{\varepsilon'}\|^2_{\tau,0}  + C\tau\|U_{\varepsilon'}\|^2_{\mathbb{H}_2} \\
&< \varepsilon'(M+M_0)+C\delta\|U_{\varepsilon'}\|^2_{\mathbb{H}_2}
 < \varepsilon .
\end{align*}

\ifdefined\xxxxxx
Since $\mathbb{H}_2$ is dense in $\mathbb{H}_0$, if we define $$\varepsilon' =\dfrac{ \varepsilon }{ 3MM_0}$$ we can find $U_{\varepsilon'} \in \mathbb{H}_2$ such that $$\|U_0-U_{\varepsilon'}\|^2_{\tau,0} <\varepsilon'.$$ 

Define $$\delta = \delta_\varepsilon = \left(
\dfrac{\varepsilon}{3C + 1}\right)^2$$ (where $C$ comes from Step 4 in the proof of part (a) considering $U_{\varepsilon'}$ as the initial condition) and notice that $\delta \to 0$ as $\varepsilon \to 0.$ 

Then, for $\tau < \delta$ we estimate, \begin{align*} \|P T^{\tau}(t) U_0\!\!- T(t)PU_0\|^2_{P(\mathbb{H}_0)} & \leqslant \|P T^{\tau}(t)(U_0\!\!-\!\!U_{\varepsilon'}) - T(t)P(U_0\!\!-\!\!U_{\varepsilon'})\|^2_{\Ha^0} + \|P T^{\tau}(t) U_{\varepsilon'}\!\!-\!\!T(t)PU_{\varepsilon'}\|^2_{\Ha^0} \\ &\leqslant \|T^{\tau}(t)\|^2_{\mathcal{L}(\mathbb{H}_0)}\|U_0\!\!-\!\!U_{\varepsilon'}\|^2_{\tau,0}  + \|T(t)\|^2_{\mathcal{L}(\Ha^0)}\|U_0\!\!-\!\!U_{\varepsilon'}\|^2_{\tau,0}  + C{\tau} ||U_{\epsilon'}||_{\tau,2}   \leqslant \varepsilon .
\end{align*}

Since $\varepsilon$ is arbitrary, the proof for finite time follows.

\fi
\

\noindent {\bf Step 2: Infinite time.} Integrating Equation \eqref{s21} in time from $s$ to $t$ we have {\small \begin{align*} I(x) &= \dfrac{\alpha}{2}\| x^\tau_t(t) \|_2^2+\dfrac{c^2}{2}\|\p x^\tau(t) \|_2^2 = \left[\dfrac{\alpha}{2}\|x^\tau_t(s)\|_2^2+\dfrac{c^2}{2}\|\p x^\tau(s)\|_2^2\right] \\ &- \delta\int_s^t \|\p x^\tau_t(\sigma)\|_2^2d\sigma  - \tau\int_s^t(u^\tau_{ttt}(\sigma),x^\tau_t(\sigma))d\sigma - \tau c^2\int_s^t (\p u^\tau_t(\sigma),\p x_t(\sigma))d\sigma \\ &\leqslant \left[\dfrac{\alpha}{2}\|x^\tau_t(s)\|_2^2+\dfrac{c^2}{2}\|\p x^\tau(s)\|_2^2\right] + \tau\int_s^t(u^\tau_{ttt}(\sigma),x^\tau_t(\sigma))d\sigma + \tau c^2\int_s^t (\p u^\tau_t(\sigma),\p x_t(\sigma))d\sigma \\ &\leqslant \left[\dfrac{\alpha}{2}\|x^\tau_t(s)\|_2^2+\dfrac{c^2}{2}\|\p x^\tau(s)\|_2^2\right]  + \tau\int_s^t[\|\sqrt{\tau}u^\tau_{ttt}(\sigma)\|_2^2+\|x^\tau_t(\sigma)\|_2^2]d\sigma \\&+ \tau c^2\int_s^t [\|\p u^\tau_t(\sigma)\|_2^2+\|\p x_t(\sigma)\|_2^2]d\sigma. \end{align*}}

Now observe that all the terms on the right hand side above (both inside and outside the integral) are uniformly exponentially stable. Therefore, there exist positive constants $L_1, L_2, a, b$ such that \begin{align*}I(x) &\leqslant \left[\dfrac{\alpha}{2}\|x^\tau_t(s)\|_2^2+\dfrac{c^2}{2}\|\p x^\tau(s)\|_2^2\right]  + \tau\int_s^\infty[\|\sqrt{\tau}u^\tau_{ttt}(\sigma)\|_2^2+\|x^\tau_t(\sigma)\|_2^2]d\sigma\\&+ \tau c^2\int_s^\infty [\|\p u^\tau_t(\sigma)\|_2^2+\|\p x_t(\sigma)\|_2^2]d\sigma\leqslant L_1e^{-as} + \dfrac{L_2}{b}e^{-bs} < \varepsilon,\end{align*} as long as $$s > T_{1,\varepsilon} \equiv \max\left\{-\frac{1}{a}\ln\left(\dfrac{\varepsilon}{2L_1}\right),-\frac{1}{b}\ln\left(\dfrac{\varepsilon b}{2L_2}\right)\right\}.$$

\

Similar estimates are valid when one integrates \eqref{s11} in time from $s$ to $t$. This will then gives rise to an $T_{2,\varepsilon}$.

Thus taking $T = T_\varepsilon \equiv \max\{T_{1,\varepsilon},T_{2,\varepsilon}\}$ in Step 1 and combining it with control of the tail of the integral leads to the convergence in \eqref{convn} uniformly for all $t \geqslant 0.$
This completes the proof of part (b) of  Theorem \ref{cr} . 

\subsection{Proof of Proposition \ref{sd}}

The uniform bounds imply, among other things, that there exist $z_1$ and $z_2$ such that \begin{equation}\begin{cases} u_t^\tau \to z_1 \ &\mbox{weakly}^\ast \ \mbox{in} \ L^\infty(0,T;\D{\p}), \\ \tau^{1/2}u_{tt}^\tau \to z_2 \ &\mbox{weakly}^\ast \ \mbox{in} \ L^\infty(0,T;L^2(\Omega)).  \end{cases}\end{equation}

An argument of Distributional Calculus shows that $u_{tt}^\tau \to z_{tt}$ in $H^{-1}(0,T;L^2(\Omega))$ and therefore $\tau^{1/2}u_{tt}^\tau \to \tau^{1/2}z_{tt} \to 0$ in $H^{-1}(0,T;L^2(\Omega).$ Uniqueness of the limit then leads to the conclusion.

\subsection{Proof of Proposition \ref{spec2}}

Recall that $\{\mu_n\}_{n\in\mathbb{N}}$ is the set of eigenvalues of $A$ and since $A$ is unbounded we can assume $\mu_n \to \infty$ as $n \to \infty.$ Proving the Proposition \ref{spec2}  amounts to the study  of  spectrum of $\A$ on the space $\mathbb{H}_0^0$.
%The whole argument basically lies in the fact that the point spectrum of $\A$ coincide withe the point spectrum of $\A^\ast$. We prove this on the lemma below.

\begin{lemma}\label{coincide_spec} 
	$$\sigma_p(\A) = \sigma_p(\A^\ast) = \{\lambda \in \mathbb{C}; \ \alpha \lambda^2 + \delta \mu_n \lambda + c^2 \mu_n = 0, \ n \in \mathbb{N}\} = \left\{\dfrac{-\delta\mu_n \pm \sqrt{\delta^2\mu_n^2-4\alpha c^2\mu_n}}{2\alpha}, \ n \in \mathbb{N}\right\}.$$ 
\end{lemma}

\begin{proof}
	Since  $A$ is  a positive self-adjoint operator with compact resolvent,  $$\sigma(A) = \sigma_p(A) \subset \mathbb{R}_+^\ast.$$ The spectrum of $A$ is countable and positive. So we assume the point spectrum is then a sequence $(\mu_n)$ such that $\mu_n \to +\infty$ as $n \to +\infty.$ We shall consider the operator $\A$ acting on $\mathcal{D}(A^{1/2}) \times \mathcal{D}(A^{1/2} ) $ with the domain $$\D{\A}= \{ (u,v) \in  \mathcal{D}(A^{1/2}) \times \mathcal{D}(A^{1/2} ); c^2 u + \delta v \in \mathcal{D}(A^{3/2}) \}.$$
	
	Let $\varphi = (\varphi_1,\varphi_2)^T \in \D{\A} .$ 
	 We seek to the describe the values of $\lambda \in \mathbb{C}$ such that \begin{equation}\label{ev}
	\A \varphi = \lambda \varphi.
	\end{equation} We compute: \begin{align*}
	\A \varphi = \left(\begin{array}{cc}
	0  & I \\
	-c^2\alpha^{-1}A & -\delta\alpha^{-1}A
	\end{array}\right)\left(\begin{array}{c}
	\varphi_1 \\
	\varphi_2
	\end{array}\right) = \left(\begin{array}{c}
	\varphi_2 \\
	-c^2\alpha^{-1}A\varphi_1 -\delta\alpha^{-1}A\varphi_2
	\end{array}\right).
	\end{align*} Therefore the equation \eqref{ev} will be satisfied if and only if $$\varphi_2 = \lambda \varphi_1$$ and $$c^2A\varphi_1 + \delta A\varphi_2 = -\alpha \lambda \varphi _2$$ which is the same as $$c^2 A \varphi_1 + \delta\lambda A\varphi_1 = -\alpha \lambda^2 \varphi_1$$ or further $$A\varphi_1 = \dfrac{-\alpha \lambda^2}{c^2 + \delta \lambda} \varphi_1.$$ The last equation means that $\varphi_1$ is an eigenvector of $A$ and because of that must be associated with some eigenvalue $\mu_n.$ Therefore, the relation between $\lambda$ and $\mu_n$ can be easily derived to be the quadratic equation $$\alpha \lambda^2 + \delta \mu_n \lambda + c^2 \mu_n = 0$$ from where follows that \begin{equation}\lambda_n = -\dfrac{\delta \mu_n}{2\alpha} \pm \dfrac{\sqrt{\delta^2\mu_n^2 - 4\alpha c^2\mu_n}}{2\alpha}.\label{qf}\end{equation}
	
	\
	
		We now characterize the point spectrum of $\A^\ast.$  Keeping in mind the following facts:
		
(i)  $A$ is self-adjoint in $H = L^2(\Omega).$ 

(ii) Fractional powers preserve self-adjointness.

	\noindent We  begin by computing  $\A^\ast.$\\
	Let $\varphi,\psi \in \D{\p} \times \D{\p}$, $\varphi = (\varphi_1,\varphi_2)^T$ and $\psi = (\psi_1,\psi_2)^T$. We have \begin{align*}
	\left(\A\varphi,\psi\right)_{ \D{\p} \times \D{\p}} &= \left(\left(\begin{array}{cc}
	0  & I \\
	-c^2\alpha^{-1}A & -\delta\alpha^{-1}A
	\end{array}\right)\left(\begin{array}{c}
	\varphi_1 \\
	\varphi_2
	\end{array}\right),\left(\begin{array}{c}
	\psi_1 \\
	\psi_2
	\end{array}\right)\right) _{ \D{\p} \times \D{\p}} \\ &=  \left(\left(\begin{array}{c}
	\varphi_2 \\
	-c^2\alpha^{-1}A\varphi_1 -\delta\alpha^{-1}A\varphi_2
	\end{array}\right),\left(\begin{array}{c}
	\psi_1 \\
	\psi_2
	\end{array}\right)\right) _{ \D{\p} \times \D{\p}} \\ &= \left(\varphi_2,\psi_1\right)_{\D{\p}} + \left(	-c^2\alpha^{-1}A\varphi_1 -\delta\alpha^{-1}A\varphi_2,\psi_2\right)_{\D{\p}} \\ &= \left(\varphi_2,\psi_1\right)_{\D{\p}} + \left(\varphi_1,-c^2\alpha^{-1} A \psi_2\right)_{\D{\p}} +  \left(\varphi_2,-\delta\alpha^{-1}A\psi_2\right)_{\D{\p}} \\ &= \left(\varphi_1,-c^2\alpha^{-1} A \psi_2\right)_{\D{\p}} +  \left(\varphi_2,\psi_1-\delta\alpha^{-1}A\psi_2\right)_{\D{\p}} \\ &= \left(\left(\begin{array}{c}
	\varphi_1 \\
	\varphi_2
	\end{array}\right),\left(\begin{array}{c}
	-c^2\alpha^{-1} A \psi_2 \\
	\psi_1-\delta\alpha^{-1}A\psi_2
	\end{array}\right)\right) _{ \D{\p} \times \D{\p}} \\ &= \left(\left(\begin{array}{c}
	\varphi_1 \\
	\varphi_2
	\end{array}\right),\left(\begin{array}{cc}
	0  & -c^2\alpha^{-1} A \\
	I & -\delta\alpha^{-1}A
	\end{array}\right)\left(\begin{array}{c}
	\psi_1 \\
	\psi_2
	\end{array}\right)\right) _{ \D{\p} \times \D{\p}}.
	\end{align*} Therefore $$\A^\ast = \left(\begin{array}{cc}
	0  & -c^2\alpha^{-1} A \\
	I & -\delta\alpha^{-1}A
	\end{array}\right),$$
	with $$\D{\A^{\ast}} = \{ (u, v) \in \mathcal{D}(A^{1/2}) \times \mathcal{D}(A^{1/2} ), c^2 u + \delta v \in \mathcal{D}(A^{3/2} ) \}. $$
	We then find out the point spectrum of $\A^\ast.$
	
	Let $\varphi = (\varphi_1,\varphi_2)^T \in \D{\A^\ast}.$ We seek to the describe the values of $\lambda \in \mathbb{C}$ such that \begin{equation}\label{ev1}
	\A^\ast \varphi = \lambda \varphi.
	\end{equation} We compute: \begin{align*}
	\A^\ast \varphi = \left(\begin{array}{cc}
	0  & -c^2\alpha^{-1} A \\
	I & -\delta\alpha^{-1}A
	\end{array}\right)\left(\begin{array}{c}
	\varphi_1 \\
	\varphi_2
	\end{array}\right) = \left(\begin{array}{c}
	-c^2\alpha^{-1}A\varphi_2 \\
	\varphi_1 -\delta\alpha^{-1}A\varphi_2
	\end{array}\right).
	\end{align*} Therefore the equation \eqref{ev1} will be satisfied if and only if $$-c^2\alpha^{-1}A\varphi_2 = \lambda \varphi_1$$ and $$\varphi_1 -\delta\alpha^{-1}A\varphi_2 =  \lambda \varphi _2.$$ 
	%which is the same as $$\lambda \varphi_1 -\delta\alpha^{-1}\lambda A\varphi_2 =  \lambda^2 \varphi _2$$ which is the same as $$-c^2\alpha^{-1}A\varphi_2 -\delta\alpha^{-1}\lambda A\varphi_2 =  \lambda^2 \varphi _2$$
	 Decoupling gives: $$A\varphi_2 = \dfrac{-\alpha \lambda^2}{c^2 + \delta \lambda} \varphi_2$$ and the last equation means that $\varphi_2$ is an eigenvector of $A$ and because of that must be associated with some eigenvalue $\mu_n$ through the quadratic equation $$\alpha \lambda^2 + \delta \mu_n \lambda + c^2 \mu_n = 0$$ which implies that $$\sigma_p(\A^\ast) = \sigma_p(\A),$$  completing the proof. 
	
	\end{proof}
	
	Part (a) then follows directly from Lemma \ref{coincide_spec} because we know that $\lambda \in \sigma_r(\A)$ if and only if $\overline{\lambda} \in \sigma_p(\A^\ast) (= \sigma_p(\A),$ in our case). However, we know that if $\overline{\lambda}\in \sigma_p(\A)$, so is $\lambda.$ Therefore, since $\sigma_p(\A) \cap \sigma_r(\A) = \emptyset$, it follows $\sigma_r(\A) = \emptyset.$
	
	Now since the parameters $\delta, \alpha, c^2 >0$ are fixed, we can see that $\sigma_p(\A)$ is eventually real, which means that no matter how we pick those parameters, since $\mu_n \to +\infty$ as $n \to \infty$ we will always be able to find an index $N$ such that from that index on all the eigenvalues will be real.
	
	It is also clear to see that the point spectrum of $\A$ is on the left side of the complex plane. In fact, it follows from the formula \eqref{qf} that in case $\lambda_n$ is complex we have $$\mbox{Re}(\lambda_n) = -\dfrac{\delta \mu_n}{2\alpha} <0.$$ 
	
	For the real ones, the $``-"$ case of the formula \eqref{qf} we have nothing to check because clearly $\lambda_n < 0.$ For the $``+"$ case we just notice that $$\sqrt{\delta^2\mu_n^2 - 4\alpha c^2\mu_n} < \delta \mu_n$$ and the strict inequality guarantees that $\lambda_n < 0$. Therefore, in order to describe the continuous spectrum of $\A$ we just analyze the limit $$\lim\limits_{n \to \infty} \dfrac{-\delta\mu_n \pm \sqrt{\delta^2\mu_n^2-4\alpha c^2\mu_n}}{2\alpha}.$$
	
	Two basic limit arguments show that $$\lim\limits_{n \to \infty} \dfrac{-\delta\mu_n + \sqrt{\delta^2\mu_n^2-4\alpha c^2\mu_n}}{2\alpha} = -\dfrac{c^2}{\delta}$$ and $$\lim\limits_{n \to \infty} \dfrac{-\delta\mu_n - \sqrt{\delta^2\mu_n^2-4\alpha c^2\mu_n}}{2\alpha} = -\infty,$$
	which implies  $\sigma_c(\A) \supset  \left\{-\dfrac{c^2}{\delta}\right\}$,  since $\sigma(\A)$ is closed and  $\sigma_r(\A) = \emptyset$  with  $-\delta^{-1}c^2$  not an eigenvalue of $\A.$
	In order to complete the proof of part (b) we need to show that $-\dfrac{c^2}{\delta} $ is {\it  the only element} in the continuous spectrum of $\A$. 
	To establish this we shall show that any $\lambda \notin \sigma_p(\A) \cup \biggl\{-\dfrac{c^2}{\delta}\biggr\}$  is in the resolvent set of $\A$. 
	Let $(f,g)^T \in \D{\p} \times \D{\p}$. We need to prove that there exists $(u,v)^T \in \D{\A}$ such that $$\A(u,v)^T - \lambda(u,v)^T = (f,g)^T.$$ 
	After writing down explicitly the equation we obtain 
	\begin{eqnarray}
	\lambda u - v = f \notag \\
	\alpha \lambda v + A( c^2 u + \delta v) = \alpha g 
	\end{eqnarray}
	which leads to solvability of 
	$$ (c^2 +\delta \lambda ) A u + \alpha \lambda^2 u= \alpha \lambda f + \alpha g + \delta A f, ~in~[ \mathcal{D}(A^{1/2})] '  $$ 
	or equivalently 
	\begin{equation}
	(c^2 +\delta \lambda ) u + \alpha \lambda^2 A^{-1}  u =  \alpha \lambda  A^{-1} f + \alpha A^{-1} g + \delta f  \equiv F
	\end{equation}
	and further because $ c^2 +\delta \lambda \ne 0, $ 
	\begin{equation}\label{res}
	u + \frac{\alpha \lambda^2}{c^2 + \delta \lambda}  A^{-1}  u=  (c^2+ \delta \lambda)^{-1} F.
	\end{equation}
	Note that $F \in \mathcal{D}(A^{1/2}) $ and we are looking for a solution $u \in \mathcal{D}(A^{1/2} )$. 
	Since $A^{-1}$ is compact on $L(\mathcal{D}(A^{1/2}))$, unique solvability of (\ref{res})  if and only if the operator 
	$$I +  \frac{\alpha \lambda^2}{c^2 + \delta \lambda}  A^{-1}$$  is injective. 
	On the other hand, the latter takes place if and only if $ - \dfrac{\alpha \lambda^2}{c^2 + \delta \lambda} \notin \sigma_p(A) $. This is also true due to the fact that $\lambda \notin \sigma_p(\A) $. 
	In view of the above, we obtain $u \in \mathcal{D}(A^{1/2} ) $ and therefore $v = \lambda u -f \in \mathcal{D}(A^{1/2})$.
	
	To conclude we need to assert that $(u,v)^T \in \mathcal{D}(\A ) $. For the latter we just notice that $ A(c^2 u + \delta v ) = \alpha g -\alpha \lambda v \in \mathcal{D}(A^{1/2} ) $-as desired by the characterization of the domain of $\A $. 
	The proof  of the Proposition is thus complete.

\subsection{Proof of Proposition \ref{spec1}} \label{secc}

\hspace{.5cm}As in the proof of Proposition \ref{spec2}, proving proposition \ref{spec1} relies on the analysis  of point spectrum of $\A^\tau$ along with the asymptotics. 
To begin with, the point spectrum of $\A^\tau$ and $(\A^\tau)^\ast$ coincide and it is given by the set $$\sigma_p(\A^\tau) = \sigma_p((\A^\tau)^\ast) = \{\lambda \in \mathbb{C}; \ \tau\lambda^3 + \alpha \lambda^2 + b^\tau\mu_n \lambda + c^2 \mu_n = 0, \ n \in \mathbb{N}\},$$ which is a consequence of basic algebraic manipulation.

The exact same argument as for Part (a) in Proposition \ref{spec2} shows Part (a) here.

Now, with empty residual spectrum we know that the points in the continuous spectrum, if any, needs  to be  in the approximate point spectrum.  By using the exact same process as in \cite{TrigMGT} (Theorem 5.2, Part ($b$), ($b_1$) and ($b_2$), p.1913 and 1914) one can show that $-\dfrac{c^2}{b^\tau}$ is an eigenvalue of $\A^\tau$ in case $\gamma^\tau = 0$ and a limit of eigenvalues in case $\gamma^\tau > 0.$ Thus  $-\dfrac{c^2}{b^\tau} \in \sigma_c(\A^\tau)$ in case $\gamma^\tau> 0.$
We shall show now that $-\dfrac{c^2}{b^\tau} $ coincides with the point  in continuous spectrum  $\sigma_c(\A^\tau)$ in case $\gamma^\tau> 0$. This is to say  $\sigma_c(\A^\tau) = \left\{-\dfrac{c^2}{b^\tau}\right\}$ . For this, it is sufficient to show that every $\lambda \in \mathbb{C} $  different from $- \dfrac{c^2}{b^{\tau} } $  and outside  the point spectrum of $\A^{\tau} $ belongs to the resolvent set. We need to prove that there exist solution $(u,v,w)^T \in \D{\A^\tau}$
As before, we consider the system: 
\begin{eqnarray}\label{cont}
v - \lambda u =f\in \mathcal{D}(A^{1/2} ) \notag \\
w -\lambda v =g \in \mathcal{D}(A^{1/2} )\notag \\
- \tau^{-1} [ c^2 Au + b^{\tau} Av + \alpha w ] - \lambda w = h \in \Ls. 
\end{eqnarray}Collecting the terms yields: 
$$(\tau^{-1} c^2 + \lambda \tau^{-1} b^{\tau} ) Au + ( \lambda^3 + \tau^{-1} \alpha \lambda^2)u = \tau^{-1}( \alpha \lambda +b^{\tau}) Af + \tau^{-1} \alpha A g +\lambda^2 f + \lambda g + h, $$ 
where the equation is defined on $[\mathcal{D}(A^{1/2})]' $. 
Since $\lambda \ne -\dfrac{c^2}{b^{\tau} } $, the above can be written as 
\begin{equation}\label{cont1}
 u + d(\lambda, \tau) A^{-1} u = F (f,g,h) \in \mathcal{D}(A^{1/2}), 
\end{equation}
where $$d(\lambda, \tau ) = \dfrac{ \lambda^3 + \tau^{-1} \alpha \lambda^2}{\tau^{-1} c^2 + \lambda \tau^{-1} b^{\tau}}  = \dfrac{\lambda^3 \tau + \alpha \lambda^2 }{c^2 + \lambda b^{\tau} }$$
and 
$$F (f,g,h)=\tau^{-1}( \alpha \lambda +b^{\tau}) Af + \tau^{-1} \alpha A g +\lambda^2 f + \lambda g + h.$$
Since $A^{-1} $ is compact  in $L(\mathcal{D}(A^{1/2})),$ unique solvability [for $u \in \mathcal{D}(A^{1/2}) $  of (\ref{cont1}) is equivalent to  the injectivity of $I + d(\lambda, \tau ) A^{-1} $. 
The latter is equivalent to the fact that $- d(\lambda, \tau) \notin \sigma_p(A^\tau),$ which in turn is equivalent to 
$\mu_n + d(\lambda, \tau) \ne 0 $. This last condition is guaranteed by the fact that $\lambda \notin \sigma_p(\A^\tau).$ Thus there exists a unique $u\in \mathcal{D}(A^{1/2}) $ 
solving (\ref{cont1}). Going back to (\ref{cont}) we obtain the improved regularity $v \in \mathcal{D}(A^{1/2} ), w\in \mathcal{D}(A^{1/2}) $ and also 
$c^2 u + b^{\tau} v \in \mathcal{D}(A) $.  Hence $(u,v,w)^T \in \mathcal{D}(\A^{\tau}) $ as desired. The proof of  equivalence $\sigma_c(\A^{\tau}) = \{ - \dfrac{c^2}{b^{\tau} } \}$ is completed.

% The first is the way done in \cite{TrigMGT} (Theorem 5.2, $b_3$, p. 1913 and 1914) by using the definition of continuous spectrum. 

In order to complete the proof of Proposition \ref{spec1} it suffices to prove the  part (c). Here, the aim is to show that when $\tau \rightarrow 0 $ the  hyperbolic branch of the spectrum of $\A^{\tau}$  escapes to $-\infty$. 
For this we show that for $n$ large, the equation \begin{equation}\label{ee} \tau \lambda^3 + \alpha \lambda^2 + (b^\tau \mu_n)\lambda + c^2 \mu_n = 0.\end{equation}  has two complex roots whose imaginary parts approach $\pm\infty.$

The argument is as follows: define a number $\theta_n$ such that $$\theta_n \approx -\frac{\gamma^\tau c^4}{\mu_n b^3} \ \mbox{for} \ n \ \mbox{large.}$$

We claim that $$-\dfrac{c^2}{b^\tau} + \theta_n \approx \lambda_n^{0,\tau} \ \mbox{for} \ n \ \mbox{large,}$$ that is, for $n$ large $-\dfrac{c^2}{b^\tau} + \theta_n$ is \emph{almost} a root of \eqref{ee}.

Indeed, notice that \eqref{ee} can be rewritten as  $$\tau \lambda^3 + \alpha \lambda^2 + (b^\tau \mu_n)\lambda + c^2 \mu_n  = \left(\lambda + \dfrac{c^2}{b^\tau} - \theta_n\right)q_n(\lambda) + r_n(\lambda),$$ where $$q_n(\lambda) = \tau \lambda^2 + (\gamma^\tau + \tau \theta_n)\lambda + b^\tau\mu_n - (\gamma^\tau + \tau\theta_n)\left(\dfrac{c^2}{b^\tau}-\theta_n\right)$$ and $$r_n(\lambda) = (\gamma^\tau + \tau\theta_n)\left(\dfrac{c^2}{b^\tau}-\theta_n\right)^2 + b^\tau\mu_n\theta_n.$$ Since $\mu_n\rightarrow\infty$ as $n\rightarrow\infty,$ we have $\theta_n\approx 0$ which implies
$$r_n(\lambda)\approx \gamma^\tau\dfrac{c^4}{b^{2\tau}}-\gamma^\tau\dfrac{c^4}{b^{2\tau}} =0 \ \mbox{for} \ n \ \mbox{large},$$
where we have used 
$b^\tau \mu_n \theta_n \approx - \dfrac{\gamma^{\tau} c^4}{b^{2\tau}} $.  
This proves the claim made above.

As a consequence,  for $n$ large we have $$\tau \lambda^3 + \alpha \lambda^2 + (b^\tau \mu_n)\lambda + c^2 \mu_n \approx \left(\lambda + \dfrac{c^2}{b^\tau} - \theta_n\right)\left(\tau \lambda^2 + (\gamma^\tau + \tau \theta_n)\lambda + b^\tau\mu_n - (\gamma^\tau + \tau\theta_n)\left(\dfrac{c^2}{b^\tau}-\theta_n\right)\right).$$ Therefore, for $n$ large the two other roots of the equation (which are complex) are approximately the two roots of $$\tau \lambda^2 + (\gamma^\tau + \tau \theta_n)\lambda + b^\tau\mu_n - (\gamma^\tau + \tau\theta_n)\left(\dfrac{c^2}{b^\tau}-\theta_n\right).$$

Then, a basic result for quadratic equation yields $$2\mbox{Re}(\lambda_{n}^{\tau,1}) = 2\mbox{Re}(\lambda_{n}^{\tau,2}) = -\dfrac{\gamma^\tau + \tau \theta_n}{\tau} \to -\dfrac{\gamma^\tau}{\tau} \ \mbox{as} \ n \to \infty$$ and
$$|\mbox{Img}(\lambda_{n}^{\tau,1})| = |\mbox{Img}(\lambda_{n}^{\tau,1})| \to +\infty \ \mbox{as} \ n \to \infty.$$

The proof of part (c) is then complete.

\section{Appendix}
\begin{lemma}(The energy identity) For all $U_0^{\tau} \in \Ha $ we have 
	\begin{equation}
	\label{eid} \dfrac{d}{dt} E_1^\tau(t) + \gamma^\tau \|u^\tau_{tt}(t)\|_2^2 = 0.
	\end{equation}
\end{lemma}
\begin{proof}
We first consider strong  solutions with initial data in $\mathcal{D}(\A^{\tau})$.  This implies   $U^{\tau}(t)  \in \Ha $  and $u_t(t), u_{tt}(t) \in \mathcal{D}(A^{1/2}) $, $u_{ttt}(t) \in H $. For these elements  the following calculus is justifiable.\\
 Notice that the expansion of $E^{\tau}_1(t)$ is
	\begin{align}\label{ef2}
	E^{\tau}_1(t)&=\frac{\tau}{2}\|u^\tau_{tt}(t)\|_2^2+\frac{b^\tau}{2}\|\p u^\tau_t(t)\|_2^2 +\frac{c^4}{2b^{\tau}}\|\p u^\tau(t)\|_2^2\nonumber\\
	& + c^2(Au^\tau_t(t),u^\tau(t))+ \frac{\tau c^2}{ b^\tau} (u^\tau_{tt}(t),u^\tau_t(t)) + \frac{\alpha c^2}{2b^\tau}\|u^\tau_t(t)\|_2^2.
	\end{align}
	Firs, taking the $L^2-$inner product of \eqref{MGTtau} with $u^\tau_{tt}$ gives
	\begin{equation}\label{eid1}
	\frac{d}{dt}\left(\frac{\tau}{2}\|u^\tau_{tt}\|^2_2+c^2(A u^\tau_t,u^\tau)+\frac{b^\tau}{2}\|\p u^\tau_t\|^2_2
	\right)+\alpha\|u^\tau_{tt}(t)\|^2-c^2(A u^\tau_t(t),u^\tau_t(t))=0.
	\end{equation}
	Next similarly, taking the $L^2-$inner product of \eqref{MGTtau} with $u^\tau_{t}$ gives
	\begin{equation}\label{eid2}
	\frac{d}{dt}\left(
	\tau(u^\tau_{tt},u^\tau_t)+\frac{\alpha}{2}\|u^\tau_t\|^2_2+\frac{c^2}{2}\|\p u^\tau\|^2_2\right)-\tau\|u^\tau_{tt}(t)\|^2_2+b^\tau\|\p u^\tau_t(t)\|^2_2=0.
	\end{equation}
	Combining $\eqref{eid1}$ and $\dfrac{c^2}{b^\tau}\times\eqref{eid2}$, we get
	\begin{align}
	\dfrac{d}{dt}\bigg[\frac{\tau}{2}\|u^\tau_{tt}\|_2^2+\frac{b^\tau}{2}\|\p u^\tau_t\|_2^2+\frac{c^4}{2b^{\tau}}\|\p u^\tau\|_2^2+c^2(A u^\tau_t,u^\tau_t)&+\frac{\tau c^2}{b^\tau}(u^\tau_{tt},u^\tau_t)+\frac{\alpha c^2}{2b^\tau}\|u^\tau_t\|_2^2\bigg]\nonumber\\ &+(\alpha-\frac{c^2\tau}{b^\tau})\|u^\tau_{tt}(t)\|_2^2=0.
	\end{align}
	By \eqref{ef2} and the definition of $\gamma^\tau=\alpha-\dfrac{c^2\tau}{b^\tau}$, we obtain the identity
	$$\dfrac{d}{dt} E_1^\tau(t) + \gamma^\tau \|u^\tau_{tt}(t)\|_2^2 = 0.$$
\end{proof}It is equivalent to say that 
\begin{equation}\label{edp}
E_1^\tau(t) + \gamma^\tau\int_0^t \|u^\tau_{tt}(s)\|^2_2ds= E_1^\tau(0),
\end{equation}
and the final conclusion is obtained by evoking density of $\mathcal{D}(\A^{\tau}) $ in  $\Ha$.

		\bibliographystyle{abbrv} 
		\bibliography{mlib2.bib}
\end{document}